\documentclass[a4paper,dvipsnames,fleqn]{article}

\usepackage{amsmath}
\usepackage{amssymb}
\usepackage{amsthm}
\usepackage[utf8]{inputenc}
\allowdisplaybreaks
\usepackage{bm}
\usepackage{enumerate}

\usepackage[british]{babel}
\usepackage{csquotes}
\usepackage[UKenglish]{isodate}
\usepackage{dsfont}
\usepackage{enumitem}
\usepackage[OT2,OT1]{fontenc}
\usepackage[cal=cm,scr=boondoxo]{mathalfa}
\usepackage{pifont}
\usepackage{stmaryrd}
\usepackage{textcomp}
\usepackage{tikz}
	\usetikzlibrary{arrows}
	\usetikzlibrary{patterns}
\usepackage{tikz-cd}
\usepackage[textwidth=3cm,colorinlistoftodos]{todonotes}
\usepackage{xfrac}
\usepackage[colorlinks=true,allcolors=purple]{hyperref}
\usepackage{cleveref}
	\crefformat{equation}{#2(#1)#3}

\numberwithin{equation}{section}			

\newcommand\cyr{
\renewcommand\rmdefault{wncyr}%
\renewcommand\sfdefault{wncyss}%
\renewcommand\encodingdefault{OT2}%
\normalfont
\selectfont}
\DeclareTextFontCommand{\textcyr}{\cyr}

\newcounter{Enum}				
\newenvironment{Enumerate}{\begin{enumerate}[label={\rm({\roman*})}]}{\end{enumerate}}

\newcommand{\descriptionlabelsave}{}		

\newcounter{StepsCount}				

\newcounter{StepsRefCount}

\theoremstyle{plain}
	\newtheorem{lemma}{Lemma}[section]
	\newtheorem{proposition}[lemma]{Proposition}
	\newtheorem{theorem}[lemma]{Theorem}
	\newtheorem{corollary}[lemma]{Corollary}
		\newtheorem{conjecture}[lemma]{Conjecture}
	\newcommand{\GenericTheoremName}{}
\theoremstyle{definition}
	\newtheorem{definition}[lemma]{Definition}
	\newcommand{\GenericDefinitionName}{}
\theoremstyle{remark}
	\newtheorem{remark}[lemma]{Remark}
	\newtheorem{example}[lemma]{Example}
	\newcommand{\GenericRemarkName}{}

\AddToHook{env/proposition/begin}{\crefalias{lemma}{proposition}}
\AddToHook{env/theorem/begin}{\crefalias{lemma}{theorem}}
\AddToHook{env/corollary/begin}{\crefalias{lemma}{corollary}}
\AddToHook{env/conjecture/begin}{\crefalias{lemma}{conjecture}}
\AddToHook{env/definition/begin}{\crefalias{lemma}{definition}}
\AddToHook{env/remark/begin}{\crefalias{lemma}{remark}}
\AddToHook{env/example/begin}{\crefalias{lemma}{example}}

\newcommand{\mc}[1]{{\mathcal{#1}}}			
\newcommand{\mf}[1]{{\mathfrak{#1}}}			
\newcommand{\bb}[1]{{\mathbb{#1}}}			
\newcommand{\wt}{\widetilde}				
\newcommand{\wh}[1]{\widehat{#1}}			

\DeclareMathOperator{\RE}{Re}				

\DeclareMathOperator{\IM}{Im}				

\DeclareMathOperator{\dom}{dom}
\DeclareMathOperator{\mul}{mul}

\DeclareMathOperator{\ran}{ran}

\DeclareMathOperator*{\wlim}{\mathit{w}-lim}
\DeclareMathOperator*{\slim}{\mathit{s}-lim}

\newcommand{\Side}[1]{\hfill{#1}\kern10pt}		
\newcommand{\FD}[5]{
	\DF\left\{\begin{array}{rcl}{#1}&\to &{#2}\\[#3pt] {#4}&\mapsto &{#5}\end{array}\right.}

\newcommand{\Dummy}{\text{\textvisiblespace\kern1pt}}	

\DeclareMathOperator{\Span}{span}			

\newcommand{\DS}{\mid\mkern3mu}				
\newcommand{\DP}{{.\kern5pt}}				
\newcommand{\DF}{\colon}				
\newcommand{\DE}{\mathrel{\mathop:}=}			
\newcommand{\ED}{=\mathrel{\mathop:}}			
\newcommand{\DD}{\mkern4mu\mathrm{d}}			

\usepackage{xcolor}
\usepackage{graphicx}
\usepackage[normalem]{ulem}

\colorlet{colorJR}{red}
\colorlet{colorAL}{blue!50}

\usepackage{todonotes} 
\colorlet{colorAL}{blue!50}
\colorlet{colorJR}{red!50}

\begin{document}

%


\begin{flushleft}
	{\Large\bf On eigenvalues of self-adjoint extensions for defect larger than one}
	\\[5mm]
	\textsc{
	Annemarie Luger, Jakob Reiffenstein
	\hspace*{-14pt}
		\renewcommand{\thefootnote}{\fnsymbol{footnote}}
		\setcounter{footnote}{2}
			\footnote{
	The second author was supported by the Sverker Lerheden foundation.}
	} \\[4ex]
    \begin{center}
    \textit{To Heinz Langer, in thankful memory}
    \end{center}
	{\small
	\textbf{Abstract.}
Self-adjoint extensions of a symmetric operator are parametrised by Krein's formula, in which the $Q$-function interacts with another analytic function (the parameter). We obtain a characterisation of the eigenvalues, isolated or not, of a given self-adjoint extension in terms of these two functions. The setting is highly general, covering symmetric operators with arbitrary defect in a Hilbert or Pontryagin space. Of independent interest is our newly developed tool, the \emph{generalised value} of a generalised Nevanlinna function, for which we give both a function-theoretic and an operator-theoretic description.
	\\[3mm]
	}
\end{flushleft}

%



%
%
%
\section{Introduction}

Finding eigenvalues of an operator often amounts to determining zeroes or singularities of analytic functions. This article provides the theoretic background for this principle for a rather large class of problems.

In many cases a self-adjoint operator appears as an extension of a symmetric operator. Typical examples for such situations are provided, for instance, by Sturm-Liouville operators or other differential operators, on intervals or more generally on graphs, see e.g., the monographs \cite{gesztesy.nichols.zinchenko:2024, kurasov:2024}.

In this situation the  (compressed)  resolvent  of a self-adjoint extension  is given by  Krein's formula, which is a kind of perturbation formula and reads as
\begin{align}
\label{eq:Krein}
P_{\mathcal H}(\wt A-\lambda)^{-1}|_{\mathcal H}=(A-\lambda)^{-1}-\gamma (\lambda)(m(\lambda)+\tau(\lambda))^{-1}\gamma(\overline{\lambda})^+. 
\end{align}
Here $A$ and $\wt A$ are both self-adjoint extensions in Hilbert spaces $\mathcal H$ and $\wt{\mathcal H} \supseteq \mc H$, respectively, of an underlying closed symmetric operator $S$ in $\mc H$ with equal deficiency indices $(n,n)$.  In this context we are  interested in the spectral properties of $\wt A$, whereas we see $A$ as the unperturbed operator, a kind of reference operator. In the examples with differential operators, $A$ might be a Dirichlet operator, whereas $\wt A$ corresponds to other self-adjoint boundary conditions. Here  eigenvalue-dependent boundary conditions  not only are allowed, but  actually are of particular interest  in the current text. 

In \eqref{eq:Krein} on the right-hand side  $\gamma$ is a  defect family and $m$  the corresponding (matrix- or operator-valued) $Q$-function of $(S,A,\gamma)$, for the precise definitions see \Cref{OS22}. At this point it is important to note that both $\gamma$ and $m$ depend on $A$ only and are analytic in the resolvent set $\varrho(A)$. 
The remaining object  in Krein's formula, the function $\tau$, is  meromorphic on $\bb C \setminus \bb R$ and  can be seen as a parameter. 

Under certain minimality conditions there is a one-to-one correspondence between the parameters $\tau$ and the extensions $\wt A$. 
Let us review two different situations. 
\begin{enumerate}
    \item {\bf Canonical extensions.}
If $\wt{\mathcal H}=\mathcal H$, i.e., the extension $\wt A$ acts in the same space as $A$, then the left-hand side  of \eqref{eq:Krein} reduces to simply the resolvent $(\wt A-\lambda)^{-1}.$  On the right-hand side this is reflected in the fact that $\tau$ is a self-adjoint constant. More precisely, in the case of defect $(1,1)$ there is a one-to-one correspondence between all self-adjoint extensions $\wt A$ in $\mathcal H$ and all $\tau\in\mathbb R\cup\{\infty\}$, here $\tau=\infty$ corresponds to the  reference extension $A$. Let us now fix some parameter $\tau$. If we  consider a point $\lambda_0$ which belongs to $\varrho(A)$, then it  belongs to the spectrum of  $\wt A$  precisely if it is a zero of $m(\lambda)+\tau$. But actually much more than that holds, namely, the spectrum of $\wt A$ coincides with the set of singularities of $-(m(\lambda)+\tau)^{-1}$. This holds also for higher defect and  is due to  the fact that the relation $\wt A$ is a minimal representing relation for the function $-(m(\lambda)+\tau)^{-1}$, cf., \Cref{O1}.

\begin{example}\label{ex:first}
As a classical example one can think of  a Sturm-Liouville differential expression
 on $[0,\infty)$  which is regular  at $0$ and in limit-point case at $\infty$.   In this case $A$ might be the corresponding self-adjoint operator whose domain consists of  functions satisfying the Dirichlet boundary condition at the left endpoint. 
The Titchmarsh-Weyl $m$-function corresponding to this boundary condition then plays the role of the $Q$-function. If $\wt A$ is any other self-adjoint realisation, defined by the boundary condition
 \[
    \cos (\theta) f(0)+\sin (\theta) f'(0)=0
 \]
 with $\theta \in (0,\pi)$, the parameter $\tau$ is constant equal to $\cot \theta$. The eigenvalues of $\wt A$ are thus precisely the generalised zeroes of $m(\lambda)+\cot \theta$.

As examples for higher (finite) defect serve differential operators on domains with more (regular) endpoints, e.g., a finite interval or a graph for defect larger than 2. 
\end{example}
\item {\bf Noncanonical extensions.} If, however, the extension $\wt A$ acts in a larger space $\wt H\supset \mathcal H$  and this space is chosen minimal (see Subsection \ref{OS22} for  details) then the situation is different. Now $\tau$ is  nonconstant and, even  though the function $-(m(\lambda)+\tau(\lambda))^{-1}$ still admits a representation involving the relation $\wt A$, this representation is not minimal in general. Hence there might be eigenvalues that cannot be seen by this $n\times n$-function. Instead one can consider the  $(2n)\times(2n)$-function
\begin{equation}\label{eq:M}
  -\begin{pmatrix}
      m(\lambda) & -I_n \\ -I_n & -\tau(\lambda)^{-1}
  \end{pmatrix}  ^{-1},
\end{equation}
which can be represented minimally with
 $\wt A$, and hence its singularities coincide with the spectrum of $\wt A$, cf., \Cref{O57}.

\begin{example}
 For a Sturm-Liouville operator as in \Cref{ex:first}  noncanonical  extensions can be seen as linearisations related to  eigenvalue-dependent boundary conditions, see  e.g.,  \cite{behrndt.jonas:2006, binding.browne.seddighi:1994, dijksma.langer.snoo:1987, dijksma.langer:1996}.   
\end{example}

\end{enumerate}
We are interested in characterizing the eigenvalues of a relation $\wt A$ given by Krein's formula \cref{eq:Krein} in terms of the $Q$-function and the parameter $\tau$, both $n\times n$-matrix functions. 
What was previously known about this question?
\begin{itemize}
    \item 
As mentioned above, for canonical extensions the parameter $\tau$ is a self-adjoint constant and the function $-(m(\lambda)+\tau)^{-1}$ has a minimal representation with $\wt A$.  Hence the eigenvalues of $\wt A$ are precisely the generalised zeroes of $m+\tau$ (a precise definition will be given in Section \ref{S3}).
\item It is also quite obvious that a point $\alpha \in\varrho(A)$ is an eigenvalue of $\wt A$ precisely if it is a  (generalised) pole  of $-(m(\lambda)+\tau(\lambda))^{-1}$, that is, a (generalised) zero of $m+\tau$. Note that in this case, by assumption, singularities on the right-hand side of \cref{eq:Krein} can only occur in  this term.
\item 
For defect $(1,1)$   noncanonical extensions are discussed in full generality in \cite{behrndt.luger:2007}. It was shown that  the eigenvalues of $\wt A$ are either generalised zeroes of $m+\tau$ or generalised poles of both $m$ and $\tau$.  
The proof required working with the $2\times2$-matrix function and to characterise its singularities, which was done using pole-cancellation functions. It was essential for the arguments used that $m$ and $\tau$ were still scalar functions. Also an example was given, which showed that the result, stated as it was, would not hold for higher defect. 
\end{itemize}

\noindent In the current paper we have now completed the picture and provide with Theorem \ref{O55} an analytic characterisation of the eigenvalues of extensions $\wt A$ even for higher defect. The main challenge is that, as $m$ and $\tau$ are now matrix functions, we have to keep track of directions (e.g. for singularities) to a completely different exten{t} than in \cite{behrndt.luger:2007}. 

To this end we have developed  the new concept of directional pairs and generalised values, which unifies and extends the notions of pole-cancellation function and generalised pole as well as root function and generalised zero.\footnote{For generalised Nevanlinna functions these have been introduced and investigated by {\bf Heinz Langer}. We  believe that he would appreciate the current development, which also clarifies old questions. }

Roughly speaking, we are dealing with functions  that admit representations including terms of the form 
\begin{align}\label{eq:intro}  
\gamma^+ (A-\lambda)^{-1} \gamma,
\end{align} cf., \Cref{prop:representation}.  In particular, this applies to the functions $m$ and $\tau$ above. If a point $\lambda$ belongs to the resolvent set of the self-adjoint relation $A$ then the function in \eqref{eq:intro} is well-defined and analytic at $\lambda$. However, we are mainly interested in real points, which might belong to the spectrum of $A$,  leaving the function a priori undefined.  There are now two ways of approaching this issue: First, one can consider  the limit behaviour of the analytic function. In order to capture the dependence on directions,  directional functions are used, which are introduced in \Cref{O31}. Second, one can use the representation and replace $\lambda$ by $\omega$ directly in 
\eqref{eq:intro}. This way one  obtains a linear relation (see \Cref{O15}), which is  possibly trivial. In  \Cref{S3} these two concepts are developed and the central result is \Cref{O33}, which shows that they lead to the same object. This provides us with  a very powerful tool, as it gives the flexibility to employ either of these two approaches, whichever is more appropriate in a given situation.

In this introduction we have motivated the questions in the case of an underlying Hilbert space, however, all our results are formulated and proven for the indefinite situation of a Pontryagin space as well. Wherever the  Hilbert space situation leads to a simplification,  this  is discussed, in particular, in Sections~\ref{sec-Nevanlinna} and \ref{sec:def}.

The article is organised as follows. We begin in \Cref{sec:prelim} by clarifying the setting.  \Cref{S3} is devoted to the new notion of  {\it generalised value}  of a generalised Nevanlinna function $Q$ and its meaning in terms of operator representations of $Q$, cf. \Cref{O33}. Moreover, we give a more explicit form of the generalised value in the definite case, i.e., for Herglotz-Nevanlinna functions.  
After these preparations, which are of interest on their own, \Cref{S4}  is devoted to the main result, \Cref{O55}. It gives a characterisation of the eigenvalues of a noncanonical self-adjoint extension $\wt A$ given by \eqref{eq:Krein} in terms of the $Q$-function $m$ and the parameter function $\tau$. We also give a simplification of the result for the definite case and discuss with several examples that the formulation is sharp otherwise. 
Finally, in \Cref{S6} we see that  all results hold also locally, i.e., if the assumptions on $Q$ and $A$ are only made locally in a neighborhood of the point of interest. 
Let us finish this introduction by pointing out that we recently learned that Seppo Hassi has also been working on this question  in the Hilbert space case. In the unpublished manuscript \cite{hassi:spexit} he has obtained a result similar to our \Cref{O55}.

\newpage

\tableofcontents

\section{Preliminaries}\label{sec:prelim}

In this section we describe the setup and collect well known results from both extension theory of symmetric operators and corresponding classes of analytic functions and their realisations. We start with the functions.

\subsection{Generalised Nevanlinna functions}

We use the notation $\mc B (\mc G,\mc H)$ for the set of bounded linear operators from a Hilbert space $\mc G$ into a Hilbert space $\mc H$ and abbreviate $\mc B (\mc G) \DE \mc B (\mc G,\mc G)$.

\begin{definition}
Let  a number  $\kappa \in \bb N \cup \{0\}$ be given, and let $Q$  be a function with  values in $\mc B (\mc G)$, which is meromorphic on $\bb C \setminus \bb R$, such that $Q(\lambda)^*=Q(\overline{\lambda})$ for all points of holomorphy of $Q$. 

The function $Q$ is said to belong to the class $\mc N_\kappa(\mc G)$ if the kernel \begin{align}
\label{O56}
K_Q(\lambda,\mu) \DE \frac{Q(\lambda)-Q(\mu)^*}{\lambda-\overline{\mu}}
\end{align}
has $\kappa$ negative squares, i.e., for any $m \in \bb N$, numbers $\lambda_1,\ldots,\lambda_m \in \bb C ^+$, and vectors $\xi_1,\ldots,\xi_m \in {\mc G}$ the matrix $\big((K_Q(\lambda_i,\lambda_j)\xi_i,\xi_j)\big)_{i,j=1}^m$ has at most $\kappa$ negative eigenvalues and $\kappa$ is minimal with this property.

A function $Q$ that belongs to a class $\mathcal N_\kappa(\mathcal G)$ for some $\kappa$ is called  a \emph{generalised Nevanlinna function (with $\kappa$ negative squares)}. Further,  by  $\mc N(\mc G)\!:=\!\bigcup\limits_{\kappa \geq0}\mc N_\kappa(\mc G)$ we denote  the set of all generalised Nevanlinna functions.

\begin{remark}
For $\kappa=0$ and $\mathcal G=\mathbb C$ these functions are the usual Herglotz-Nevanlinna functions, i.e., functions that holomorphically map the complex upper halfplane $\mathbb C^+$ into the closed upper halfplane $\mathbb C^+\cup\mathbb R$.
\end{remark}
\end{definition}
We give now several examples,  which we will return to later. 
\begin{example}\label{ex:N-null}
To start with, we have the scalar   $\mathcal N_0$-functions  
$$
Q_1(\lambda):=-\frac{1}{\lambda+1}+\lambda \quad \text{ and }\quad  Q_2(\lambda):=\sqrt{2\lambda}  \quad \text{ for  }\lambda\in\mathbb C^+
$$
as well as 
$$ 
Q_3(\lambda):=3i \quad \text{ and }\quad  Q_4(\lambda):= -\frac{1}{\lambda+4i} \quad \text{ for  }\lambda\in\mathbb C^+.
$$
Note that, by definition,
$$ 
 Q_3(\lambda)=-3i \quad \text{ and }\quad  Q_4(\lambda)= -\frac{1}{\lambda-4i} \text{ for  }\lambda\in\mathbb C^-.
$$
\end{example}
\begin{example}\label{ex:N-null-matrix}
Also the matrix-valued function
$$ 
 Q_5(\lambda):=\begin{pmatrix}
     -\frac1\lambda & 5 \\[1mm] 5 & \lambda
 \end{pmatrix}
$$
is a Herglotz-Nevanlinna function with no negative squares.
\end{example}
\begin{example}\label{ex:N-kappa}
Both the functions  
$$
Q_6(\lambda):=\frac{1}{\lambda}+6  \quad \text{ and }\quad  Q_7(\lambda):=\lambda^2\cdot  \sqrt{7\lambda}, 
$$
however, have  one negative square. This follows from the general fact that a function is a generalised Nevanlinna function if and only if it is of the form $Q(\lambda)=R(\overline\lambda)^*Q_0(\lambda) R(\lambda)$, where $R$ is a rational function and $Q_0\in\mathcal N_0$, see \cite{dijksma.langer.luger.shondin:2000, derkach.hassi.snoo:1999,luger:2002} for the scalar and the matrix case, respectively.
\end{example}

We denote by $\mf h(Q)$ the \emph{extended domain of holomorphy} of $Q$, i.e., the set consisting of  points in $\bb C \setminus \bb R$ where $Q$ is holomorphic  
as well as all points  $\mu \in \bb R$ such that $Q$ has a holomorphic continuation to a neigbourhood of  $\mu$.

\begin{example}
We have $\mf h(Q_1)=\mathbb C\setminus \{-1\}$ but   $\mf h(Q3)=\mf h(Q4)=\mathbb C\setminus \mathbb R$, as both functions have a jump in the imaginary part along the real line.   
\end{example}

Generalised Nevanlinna functions are closely related to self-adjoint operators in Pontryagin spaces.
In particular, the following characterisation holds. See e.g., \cite{behrndt.luger:2007} for a version with rather general assumptions as well as for references to earlier results.

\begin{proposition}\label{prop:representation}
A function $Q$ meromorphic in $\mathbb C\setminus\mathbb R$ with values in $\mathcal B(\mathcal  G)$ is a generalised Nevanlinna function if and only if there exist  a self-adjoint linear relation $A$ in a Pontryagin space $\Pi$, a bounded linear map $\gamma$ from {$\mc G$} to $\Pi$, and  a point $\lambda_0 \in \mf h(Q)\cap \rho(A) \cap \mathbb C^+$ such that for $\lambda \in \mf h(Q) \cap \rho (A)$ it holds
\begin{align}
\label{O28}
Q(\lambda)=Q(\lambda_0)^*+(\lambda-\overline{\lambda_0})\gamma^+ \big(I+(\lambda-\lambda_0)(A-\lambda)^{-1} \big)\gamma.
\end{align}
Here $\gamma^+$ denotes the adjoint of $\gamma$ with respect to the inner product on $\Pi$. Moreover, the representation \eqref{O28} can be chosen minimal, i.e., such that
\begin{align}
\label{O29}
\Pi= \overline{\Span} \big\{{\big(I+(\lambda-\lambda_0)(A-\lambda)^{-1} \big)\gamma} x \DS \lambda \in \rho(A), \, x \in {\mc G} \big\}.
\end{align}
\end{proposition}

\begin{remark}
\label{O21}
Given $\Pi$, $A$, and $\gamma$, formula \eqref{O28} defines a function $Q$ uniquely up to a self-adjoint additive constant. Note that \eqref{O28} forces $\IM Q(\lambda_0)=(\IM \lambda_0) \gamma^+ \gamma$ while $\RE Q(\lambda_0)$ can be chosen at will.
\end{remark}
\begin{remark}
If the representation \eqref{O28} is minimal then $Q\in\mathcal N_\kappa{(\mathcal G)}$ if and only if $\kappa$ is the index of negativity of the Pontryagin space $\Pi$, i.e., the maximal dimension of a negative definite subspace of $\Pi$. In particular, if $\kappa=0$, i.e., $Q$ is a Herglotz-Nevanlinna function, then the space $\Pi$ can be chosen as a Hilbert space. 
\end{remark}

\Cref{prop:representation} gives rise to the following notation. 

\begin{definition}
Let $Q \in \mc N(\mc G)$ be represented as in \eqref{O28}, where $\Pi$ is a Pontryagin space, $A$ is a self-adjoint linear relation in $\Pi$, and $\gamma: {\mc G} \to \Pi$ is a bounded linear map. Assume that the minimality condition \eqref{O29} holds. Then $(\Pi,A,\gamma)$ is called a \emph{minimal realisation} of $Q$. 
\end{definition}
\begin{remark} This property is important for us since 
minimality of a realisation $(\Pi,A,\gamma)$ implies $\mf h(Q)=\rho(A)$, cf., \cite{krein.langer:1973}. 
\end{remark}

For later use we introduce the notation
\begin{align}
\label{O58}
\gamma (\lambda) &\DE (I+(\lambda-\lambda_0)(A-\lambda)^{-1})\gamma, &&\lambda \in \rho (A).
\end{align}
Then $\gamma (\lambda)$ and $\gamma (\mu)$ are related by
\begin{align}
\label{O62}
\gamma (\mu) &= (I+(\mu-\lambda)(A-\mu)^{-1})\gamma (\lambda), &&\mu,\lambda \in \rho (A)
\end{align}
and clearly we have $\gamma (\lambda_0)=\gamma$. We also note that the representation \eqref{O28} can be written as
\begin{align}
\label{O59}
Q(\lambda)=Q(\lambda_0)^*+(\lambda-\overline{\lambda_0})\gamma^+ \gamma (\lambda)
\end{align}
and that it holds
\begin{align}
    \label{O93}
    K_Q(\lambda,\mu)=\frac{Q(\lambda)-Q(\mu)^*}{\lambda-\overline{\mu}}=\gamma(\mu)^+ \gamma (\lambda), \qquad \lambda, \mu \in \mf h(Q).
\end{align}

For later use let us exemplify these notions for an easy example. 
\begin{example}\label{ex:i}
 The scalar $\mathcal N_0$-function 
    $$Q_{{3}}(\lambda)=\left\{ \begin{array}{rc}
        {3}i,  &\lambda \in\mathbb C^+\\
        -{3}i, &\lambda \in\mathbb C^-
    \end{array} 
    \right.
    $$
    has a minimal  realisation in the Hilbert space  $L^2(\mathbb R,\sigma)$, where $\sigma$ denotes the scaled Lebesgue measure $d\sigma(t)=\frac{{3}}\pi d t$. Here $A$ can be chosen as the multiplication operator by the independent variable, i.e., 
    $$(A x)(t)=t\cdot x(t), \text{ with  }\gamma:=\frac1{t-\lambda_0}\text{ and  }\gamma(\lambda):=\frac1{t-\lambda}.
    $$
    Note that the function $Q_{{3}}$ cannot be extended holomorphically to any  real point and  $\rho(A)=\mathbb C\setminus \mathbb R=\mf h(Q_{{3}})$. 
\end{example}

In the following, together with a generalised Nevanlinna function we will also consider its pointwise inverse.

\begin{definition}
A generalised Nevanlinna function $Q$ is called \emph{regular} if there exists $\mu_0 \in \mf h (Q)$ such that $Q(\mu_0)$ is boundedly invertible. In this case  set
\begin{align*}
\wh Q(\lambda) \DE -Q(\lambda)^{-1}
\end{align*} 
for all $\lambda \in \mf h(Q)$ with $0 \in \rho (Q(\lambda))$.
\end{definition}

The kernels of $\wh Q$ and $Q$ are related by
\begin{align}
\label{O61}
&K_{\wh Q}(\lambda,\mu)=\wh Q(\mu)^* K_Q(\lambda,\mu) \wh Q(\lambda), \\
\label{O54}
&\big(K_{\wh Q}(\lambda,\mu)\xi_{\lambda},\xi_{\mu} \big)=\big(K_Q(\lambda,\mu)\wh Q(\lambda)\xi_{\lambda},\wh Q(\mu)\xi_{\mu}\big)
\end{align}
for $\lambda,\mu \in \mf h(Q) \cap \mf h(\wh Q)$ and arbitrary vectors $\xi_{\lambda},\xi_{\mu} \in \mc G$. Thus $\wh Q$ is again a generalised Nevanlinna function. It has a minimal realisation of the same form as \eqref{O59}, and one such realisation is given in the following lemma in terms of a minimal realisation of $Q${, see \cite[Proposition 2.1]{luger:2002} or \cite[Theorem 3.1]{emmel.luger:2023}. 

\begin{lemma}
\label{O1}
Let $Q \in  \mc N(\mc G)$ be regular, and let $(\Pi,A,\gamma )$ be a minimal realisation of $Q$. Fix $\lambda_0 \in \mf h (Q)$ such that  $0 \in \rho (Q(\lambda_0))$, and define the relation $\wh{A}$ by
\begin{align}
\label{O24}
(\wh{A}-\lambda_0)^{-1} \DE (A-\lambda_0)^{-1}+\gamma(\lambda_0)\wh Q(\lambda_0)\gamma(\overline{\lambda_0})^+.
\end{align}
Further let $\hat{\gamma} \DE \gamma \wh Q(\lambda_0)$ and
\begin{align}
\hat{\gamma}(\lambda) \DE \big(I+(\lambda-\lambda_0)(\wh{A}-\lambda)^{-1} \big)\hat{\gamma}, \qquad \lambda \in \rho (\wh{A}).
\end{align}
Then $\wh Q$ is in  $\mc N(\mc G)$ and it has the minimal realisation $(\Pi,\wh{A},\hat{\gamma})$. The identities
\begin{align}
(\wh{A}-\lambda)^{-1} = (A-\lambda)^{-1}+\gamma(\lambda)\wh Q(\lambda)\gamma(\overline{\lambda})^+, \qquad \hat{\gamma}(\lambda) = \gamma(\lambda)\wh Q(\lambda)
\end{align}
hold for all $\lambda \in \mf h (Q) \cap \mf h (\wh Q)$.
\end{lemma}

\subsection{$Q$-functions and Krein's formula}
\label{OS22}

In this subsection we recall some classical notions and results from the extension theory  of symmetric operators. See{,} e.g., \cite{derkach:1999, dijksma.snoo:1987,dijksma.snoo:1987b}  for details.

Let $S$ be a closed symmetric operator with finite  or infinite defect $(n,n)$ in a Pontryagin space $\Pi$. Then there exists at least one self-adjoint extension. We fix such a self-adjoint extension $A$ of $S$ in $\Pi$ and a point $\lambda_0\in\rho(A)\cap\mathbb C^+$. Moreover, choose a Hilbert space $\mc G$ with $\dim\mc G=n$ and fix a linear bijection $\gamma \in \mc B(\mc G,\ker (S^+-\lambda_0))$. Define the \emph{defect family} $\gamma (\lambda)$ as in \eqref{O58}. 
We further assume that $S$ is \emph{simple}, i.e., the following minimality condition holds.
\begin{align*}
\Pi = \overline{\Span} \{\ker (S^+-\lambda) \DS \lambda \in \rho(A) \}=\overline{\Span} \{\gamma (\lambda)x \DS \lambda \in \rho(A), \, x \in \mc G \}.
\end{align*}
Then there exists a function $m \in \mc N(\mc G)$ such that
\begin{align}
\label{O48}
\frac{m(\lambda)-m(\mu)^*}{\lambda-\overline{\mu}}=\gamma(\mu)^+ \gamma (\lambda), \qquad \lambda, \mu \in \rho (A);
\end{align}
this $m$ is called \emph{$Q$-function} of the triple $(S,A,\gamma)$. It is uniquely determined up to a self-adjoint additive constant, and explicitly given by
\begin{align}
\label{O49}
m(\lambda) =m(\lambda_0)^* + (\lambda- \overline{\lambda_0})\gamma^+ \big(I +(\lambda-\lambda_0)(A-\lambda)^{-1} \big)\gamma
\end{align}
for $\lambda \in \rho (A)$. Hence a minimal realisation of $m$ is given by $(\Pi,A,\gamma)$.

\begin{remark}
    We use here the notion $Q$-function (as e.g., in \cite{krein.langer:1973}), but {note that $m$ is nothing but the} \emph{Weyl function} {of an associated boundary triple}, see e.g., {\cite[Theorem~1]{derkach:1994}}.
\end{remark}

Note that the function $m$ defined by \eqref{O49} is a generalised Nevanlinna function, and it  is \emph{strict}, i.e., for any $\mu \in \rho(A)$ we have
\begin{align}
\label{O41}
\bigcap_{\lambda \in \rho(A)} \ker \frac{m(\lambda)-m(\mu)^*}{\lambda-\overline{\mu}}=\{0\}.
\end{align}
This holds because of $\ker \gamma =\{0\}$. 
Strictness is thus a necessary condition for $m$ to be a $Q$-function of some triple $(S,A,\gamma)$. It is well-known that it is also a sufficient condition.

In this setting, just like for $S$ acting in a Hilbert space, Krein's formula describes all its self-adjoint extensions. The parameter $\tau$ is a generalised Nevanlinna family, that is, a relation-valued generalised Nevanlinna function. We refrain from giving the precise definition, as we will not use it in the following, but the interested reader can find it in \cite[Definition~2]{derkach:1994}.

\begin{theorem}[{\cite[Theorem~3]{derkach:1994}}]
Let $\Pi$, $S$, $A$, and $\gamma(\lambda)$ as above and $m$ a corresponding $Q$-function.
 Then for any self-adjoint extension $\wt A$ of $S$ in a Pontryagin space $\widetilde\Pi\supset \Pi$ there is a unique generalised Nevanlinna family $\tau$ with $0 \in \rho (m(\lambda_0)+\tau (\lambda_0))$ such that
\begin{align}
\label{eq:kreinII}
P(\wt A-\lambda)^{-1}|_{\Pi}=(A-\lambda)^{-1}-\gamma (\lambda)(m(\lambda)+\tau(\lambda))^{-1}\gamma(\overline{\lambda})^+,
\end{align}
where $ P$ denotes the orthogonal projection from $\widetilde\Pi$ into $\Pi$.
Conversely, for every generalised Nevanlinna family $\tau$ with $0 \in \rho (m(\lambda_0)+\tau (\lambda_0))$ there is a self-adjoint extension $\wt A$ of $S$ in a Pontryagin space {$\widetilde\Pi\supset$} $\Pi$ such that \eqref{eq:kreinII} holds. Moreover, the extension $\wt A$ can be chosen to be $\Pi$-minimal, i.e., 
\begin{align}
\label{eq:krein-minimal}
\wt\Pi=\overline{\Span} \big\{(I+(\lambda-\lambda_0)(\widetilde A-\lambda)^{-1})\gamma \DS \lambda \in \rho (\widetilde A), \, x \in {\mc G} \big\},
\end{align}
and any two $\Pi$-minimal extensions satisfying \eqref{eq:kreinII} are unitarily equivalent.

\end{theorem}

\begin{remark}
In the case of $\Pi$-minimality the sum of the negative indices of $m$ and $\tau$ equals the negative index of the extension space $\widetilde\Pi$. Moreover, the parameter $\tau$ is a constant if and only if the extension is canonical, i.e., $\widetilde \Pi=\Pi$.    
\end{remark}

\begin{remark}
    In \Cref{O55} we shall give a complete description of the eigenvalues of $\wt A$, assuming that $\tau$ is a generalised Nevanlinna function.
    Note that if $\tau$ is not single-valued, then the perturbation term in \eqref{eq:kreinII} has a kernel, which means that the extensions $A$ and $\wt A$ have more in common than only $S$. In this sense, for our purpose, it is enough to consider parameters $\tau$ which are functions, and we do not need Nevanlinna families in full generality.
\end{remark}

\section{Directional functions and generalised values}
\label{S3}

Generalised Nevanlinna functions are by definition meromorphic on $\mathbb C\setminus \mathbb R$. A point on the real line, however, might belong to the extended domain of holomorphy, $\mf h(Q)$, or it can be  an (isolated) pole or a nonisolated singularity. 

\begin{example}
   For the function 
    $Q_{3}$ 
    from  \Cref{ex:N-null} the (nontangential) limit 
     at the point $\lambda=0$ (as well as at every other real point) is nonreal. 
\end{example}
\begin{example}
    The function $Q_1$ from \Cref{ex:N-null} obviously has a limit at $\lambda=0$, as $Q_1$ is analytic at  this point. Also the $\mathcal N_0$-function
    $$
    Q_8(\lambda):=\frac{\lambda}{8-i\lambda}=-\left(-\frac8\lambda+i\right)^{-1}
    $$
    has a real limit as $\lambda$ tends to $0$ (nontangentially). However, at all other real points the limit is nonreal. 
\end{example}

\begin{example}
    Also the $\mathcal N_0$-function $Q_2(\lambda)=
    \sqrt{2\lambda}$ (with branch cut along the negative real line)
     has at the point $\lambda=0$ a nontangential limit that is real, however, it is weaker as we will see below. 
\end{example}

\begin{example}
    The matrix function $Q_5$ from \Cref{ex:N-null-matrix} has at $\lambda=0$ a singularity in one direction, but has a limit in the orthogonal direction.  
\end{example}

In the following we are going to  make this observation precise.

\subsection{Directional functions and directional pairs}

For a function $Q \in \mc N(\mc G)$ and a real point $\omega$ we introduce  the notion of $Q$ having a \emph{generalised value}. This unifies and generalises the well-established notions of {generalised poles and zeroes}.

\begin{definition}
\label{O31}
Let $Q \in \mc N(\mc G)$ and ${\omega} \in \bb R$ be given {and} let $\mc U_{{\omega}}$ be a neighborhood of ${\omega}$.
Suppose $f : \, \mc U_{{\omega}} \cap \, \mf h(Q) \, \cap \bb C^+ \to \mc G$ is an analytic function {and denote $g(\lambda)\DE Q(\lambda)f(\lambda)$. If }  the nontangential limits\footnote{
The equality $\lim_{\lambda \hat{\to} {\omega}} f(\lambda)={f_0}$ means that for every angle $\theta \in (0,\frac{\pi}{2})$ and corresponding angular domain $W_{\omega,\theta} \DE \{{\omega} + re^{i\phi} \DS r>0, \, |\phi-\frac{\pi}{2}| \leq \frac{\pi}{2}-\theta \}$ we have \[\lim_{\substack{\lambda \to {\omega} \\ \lambda \in W_{{\omega},\theta}}} f(\lambda)={f_0}. \]
}
\begin{align}\label{eq:diff}
& \lim_{\lambda \hat{\to} {\omega}} f(\lambda){=: f_0}, \qquad  \lim_{\lambda \hat{\to} {\omega}} Q(\lambda)f(\lambda){=\lim_{\lambda \hat{\to} {\omega}} g(\lambda)=:g_0}
\end{align}
exist, and 
\begin{align}
\label{O37}
&\bigg(\frac{Q(\lambda)-Q(\lambda)^*}{\lambda-\overline{\lambda}}f(\lambda),f(\lambda) \bigg)_{\mc G} \quad \text{is bounded as }\lambda \hat{\to} {\omega},
\end{align}
then the function $f$ is called a \emph{directional function} of $Q$ at ${\omega}$. In this case, we call $({f_0};{g_0})$ a \emph{directional pair} of $Q$ at ${\omega}$. If there exists a nontrivial directional function at $\omega$ (i.e., the corresponding directional pair is nontrivial), then we say that the function $Q$ has a nontrivial \emph{partial generalised value} (or simply \emph{generalised value}) at ${\omega}$.

\end{definition}

Note that, by definition, the trivial function $f\equiv0$ is a directional function, and the trivial pair $(0;0)$ is a directional pair for every $Q \in \mc N(\mc G)$ and every $\omega \in \bb R$. The reason that these uninteresting cases are included here, will be more evident when we deal with  the set of all directional pairs as a linear relation.
Let us now go back to the examples from the beginning of this section. 
\begin{example}
For the functions $Q_1$ and $Q_8$ the constant function $f(\lambda):=1$ is a directional function and gives the directional pairs $(1;-1)$ and $(1;0)$, respectively. 
\end{example}

\begin{example}\label{ex:zero-pole}
For the matrix function $Q_5$  every vector function of the form $f(\lambda):=\binom{a\lambda}{b}$ with $a,b\in\mathbb C$ serves as a directional function. The corresponding directional pair is then $\big(\binom 0b{;} \binom{5b-a}{0} \big)$.
Note, in particular, that in order to reach the directional pair $\left(\binom 01 {;}\binom 00 \right)$ the directional function cannot be a constant. 
\end{example}
\begin{example}
The function $Q_3$ (as well as $Q_4$) does not have a nontrivial generalised value at any real point, as, due to the nonreal limit, condition \cref{O37} cannot be satisfied. 
\end{example}
\begin{example}
Even though  the function $Q_2$ has a real limit at $\lambda=0$, there does not exist a nontrivial directional function at this point.  

The function $Q_7$ does have a generalised value at $\lambda =0$, as directional function serves the constant function $f(\lambda)=1$.
\end{example}
We also want to point out that directional functions cannot always  be chosen polynomially, even for rational functions in $\mc N_0(\mc G)$, cf.,  \Cref{ex:ell2}.

\begin{remark}\label{re:earlier}
Note that  Definition \ref{O31} also includes some well-established notions:
\begin{itemize}
\item 
If $f_0\neq0$   and $g_0=0$, then  $f$ is  a root function of $Q$ at $\omega$ with root vector $f_0$. In this case $\omega$ is a so-called generalised zero of $Q$, see e.g., \cite{borogovac.langer:1988,borogovac.luger:2014}
\item 
If $f_0=0$   and $g_0\neq0$, then  $f$ is  a pole-cancellation function  of $Q$ at $\omega$ with pole vector $g_0$, and $\omega$ is a   so-called a generalised pole of $Q$.
\item In \cite{behrndt.luger:2007} the notion {``}generalised value{''} was already used in the special case when the direction is not needed, i.e. if every constant function $f(\lambda)\equiv f_0$ is a directional function of $Q$ at $\omega$ then it was said that $Q$ attains a generalised value at $\omega$, cf., \cite{behrndt.luger:2007}.
\end{itemize}
The existence of a nontrivial directional function and hence of a nontrivial directional pair $(f_0;g_0)$ can thus be interpreted as {$Q$ attaining} a kind of directional value $g_0$  in direction $f_0$. 
\end{remark}

Originally, generalised poles and zeroes were defined as eigenvalues of the representing relation of $Q$ and $\widehat Q$, respectively. Later these operator-theoretically defined points have been characterised analytically, in terms of the existence of pole-cancellation functions and root functions. See \cite{borogovac.langer:1988}
for early results in this direction and \cite{borogovac.luger:2014} for a recent discussion (with references) and a comparably easy construction. However, our present approach is the opposite of the historical one, i.e.,  we choose to start  with the analytic definition and will give other characterisations later on, in particular,  in Theorem \ref{O33}.

As a first step in this direction, we construct for each directional pair $({f_0};{g_0})$ a vector ${v} \in \Pi$ which is related to $({f_0};{g_0})$ through equations \eqref{O44}, \eqref{O46}. In order to prove the following proposition, we make minor adjustments to the proof of \cite[Theorem~3.3]{borogovac.luger:2014}.

\begin{proposition}\label{O42}
Let $Q \in \mc N(\mc G)$ with the minimal realisation $(\Pi,A,\gamma)$ be given and let $\omega \in \bb R$. Suppose that $({f_0};{g_0})$ is a directional pair of $Q$ at $\omega$. Then there exists a vector ${v} \in \Pi$ such that 
\begin{align}
\label{O44}
\gamma {f_0} =\big(I+(\lambda_0-\omega)(A-\lambda_0)^{-1} \big){v}, \\
\label{O46}
{g_0} = Q(\lambda_0)^* {f_0} +(\omega-\overline{\lambda_0}) \gamma^+ {v}.
\end{align}
If  $({f_0};{g_0})$ has the  directional function  $f$ then such a vector $v$ can  be obtained  as the weak limit 
$$
\gamma (\lambda) f(\lambda)\rightharpoonup v
$$
as $\lambda \hat{\to} \omega$.
\end{proposition}

\begin{proof}
Recall that a sequence $(x_k)_{k \in \bb N}$ in the Pontryagin space $(\Pi,[.,.])$ converges weakly if and only if $([x_k,x_k])_{k \in \bb N}$ is bounded and $([x_k,u])_{k \in \bb N}$ is a Cauchy sequence for all $u$ in a subset whose linear span is dense in $\Pi$. 
We use this statement to prove weak convergence of $\gamma (\lambda) f(\lambda)$ for $\lambda \hat{\to} \omega$.

First note that
\begin{align*}
[\gamma (\lambda) f(\lambda),\gamma (\lambda) f(\lambda)]_{\Pi}=\big(K_Q(\lambda,\lambda)f(\lambda),f(\lambda) \big)_{\mc G}
\end{align*}
is bounded for $\lambda \hat{\to} \omega$. Second, for $\mu \in \rho (A) \cap \bb C^+$ and $h \in \mc G$ we have
\begin{align*}
[\gamma (\lambda) f(\lambda),\gamma (\mu)h]_{\Pi}=\bigg(\frac{Q(\lambda)-Q(\mu)^*}{\lambda-\overline{\mu}}f(\lambda),h \bigg)_{\mc G} \hspace{-4pt} \xrightarrow{\lambda \hat{\to} \omega} \frac{\big({g_0}-Q(\mu)^* {f_0} ,h \big)_{\mc G}}{\omega -\overline{\mu}}.
\end{align*}
Since $\Span \{\gamma(\mu)h  \,| \, \mu \in \rho (A) \setminus \{\omega \}, \, h \in \mc G \}$ is dense in $\Pi$, we indeed get that $\gamma (\lambda) f(\lambda)$ converges weakly to some element ${v} \in \Pi$ as $\lambda \hat{\to} \omega$. Using \eqref{O62} we have
\begin{align*}
\gamma {f_0} &=\gamma (\lambda_0){f_0}=\wlim_{\lambda \hat{\to} \omega}  \big(I+(\lambda_0 - \lambda)(A-\lambda_0)^{-1} \big)\gamma (\lambda)f(\lambda) \\
&=\big(I+(\lambda_0-\omega)(A-\lambda_0)^{-1} \big){v}.
\end{align*}
In order to see the additional statement we compute
\begin{align*}
{g_0} &=\wlim_{\lambda \hat{\to} \omega} Q(\lambda)f(\lambda)=\wlim_{\lambda \hat{\to} \omega} \big[Q(\lambda_0)^*+(\lambda-\overline{\lambda_0})\gamma^+ \gamma (\lambda) \big] f(\lambda) \\
&=Q(\lambda_0)^* {f_0}+(\omega - \overline{\lambda_0})\gamma ^+ {v},
\end{align*}
which finishes the proof.
\end{proof}

As one of the main results of this section we will show  that also the inverse of \Cref{O42} holds.  That is, $(f_0;g_0)$ is a directional pair precisely if there exists $v \in \Pi$ such that \eqref{O44} and \eqref{O46} hold true. In particular, this will imply that the sum of two directional pairs is again a directional pair.

\begin{remark}
\label{O39}
    In fact, if $f$ and $\tilde f$ are directional functions of $Q \in \mc N(\mc G)$ at $\omega \in \bb R$ then $f+\tilde f$ is again a directional function. Yet, due to the nonlinearity of the kernel condition \eqref{O37}, this statement is not trivial.
    At this point we are not yet able to prove it, but it will follow from later results. The idea is as follows: First the statement is proved for the special case  $Q \in \mc N_0(\mc G)$ in \Cref{O71}. For the general situation we find in \Cref{O25} that there always exists a M\"obius transform $\wt Q$ of $Q$ which does not have a generalised pole at $\omega$. The function $\wt Q$ can be written as the sum of a function in $\mc N_0(\mc G)$ and a function holomorphic at $\omega$, cf. \cite[Proposition~3.3]{daho.langer:1985}. Then the special case implies that the statement holds for the M\"obius transform $\wt Q$ and thus, by \Cref{O32}, also for the original function $Q$. 
\end{remark}

\subsection{The generalised value  $Q(\omega)$}

In order to capture better the behaviour of $Q$ at real points even outside $\mathfrak{h}(Q)$  we introduce one more object which, as we will see in \Cref{O33}, is closely related to directional pairs.

Note, if a generalised Nevanlinna function $Q$ has a minimal realisation with self-adjoint relation $A$ then the extended domain of holomorphy $\mathfrak{h}(Q)$   coincides with the resolvent set $\rho(A)$ and $Q$ 
is still well-defined via the operator representation \eqref{O28} in these points. However, even if $\omega\notin\rho(A)$ the inverse $(A-\omega)^{-1}$ is defined as a relation (with possibly small domain). This allows us to "evaluate $Q$ at $\omega$", which leads to the following definition.

\begin{definition}
\label{O15}
Let $Q \in \mc N(\mc G)$ and fix a minimal realisation $(\Pi,A,\gamma)$ of $Q$. 
For every $\omega \in \bb R$ define the linear relation $Q(\omega)\in\mc G\times \mc G$ by
\begin{align}
\label{O76}
Q(\omega) \DE Q(\lambda_0)^*+(\omega-\overline{\lambda_0})\gamma^+\big(I+(\omega-\lambda_0)(A-\omega)^{-1} \big) \gamma.
\end{align}
We refer to  $Q(\omega)$ as the \emph{partial generalised value} (or simply, \emph{generalised value}) of $Q$ at $\omega$.
\end{definition}

\begin{remark}
    At the moment it is not at all obvious that 
    the notion of $Q$ "having a generalised value at $\omega$" in \Cref{O31} fits with this notion of the actual object  "generalised value" for the relation $Q(\omega)$. However, \Cref{O33} will justify these notions.   
\end{remark}
We first show that $Q(\omega)$ is well{-}defined.

\begin{lemma}
\label{O91}
Let $Q \in \mc N(\mc G)$ and $\omega \in \bb R$. Then the linear relation   $Q(\omega)$ given by \Cref{O15} does not depend on the minimal realisation of $Q$. Moreover, $Q(\omega)$ is symmetric in $\mc G$.
\end{lemma}
\begin{proof}
Given two minimal realisations $(\Pi,A,\gamma)$ and $(\wt{\Pi},\wt A,\wt{\gamma})$ of $Q$, the assignment 
\begin{align*}
    \gamma (\lambda)x \mapsto \wt{\gamma}(\lambda)x, \qquad \lambda \in \rho (A)=\rho(\wt A), \quad x \in \mc G
\end{align*}
defines an isometric operator on the dense subspace $\Span \{\gamma (\lambda)x \DS \lambda \in \rho(A), \, x \in \mc G \}$ with range $\Span \{\wt{\gamma} (\lambda)x \DS \lambda \in \rho(\wt A), \, x \in \mc G \}$ \cite[p.~145]{langer.textorius:1977}. It can be extended to a unitary operator $U:\Pi \to \wt{\Pi}$ satisfying $UA=\wt A U$. Observe that this implies $U(A-\omega)^{-1}=(\wt A-\omega )^{-1}U$ in the sense of linear relations. We obtain
\begin{align*}
  &\gamma^+\big(I+(\omega-\lambda_0)(A-\omega)^{-1} \big) \gamma = (U\gamma)^+ U\big(I+(\omega-\lambda_0)(A-\omega)^{-1} \big) \gamma \\
  &= (U\gamma)^+ \big(I+(\omega-\lambda_0)(\wt A-\omega)^{-1} \big) U\gamma =\wt{\gamma}^+\big(I+(\omega-\lambda_0)(\wt A-\omega)^{-1} \big) \wt{\gamma}
\end{align*}
and thus \Cref{O15} defines $Q(\omega)$ independently of the chosen minimal realisation. It remains to note that $Q(\omega)$ is symmetric due to the computation rule $T^*S^* \subseteq (ST)^*$ for linear relations $S,T$.
\end{proof}

Note, if $\omega \in \rho (A)$ then  $Q$ has a real-analytic continuation to a neighborhood of $\omega$, whose value at $\omega$ coincides with $Q(\omega)$ (identifying operators with their graphs). In this case $\omega\in\mf h(Q)$. However, {even} {if $\omega \notin\rho(A)$}  the function $Q$ (restricted to the upper half plane $\mathbb C^+$)  might {still} have an analytic continuation $\wt Q$ to a neighborhood of $\omega$. In that case the graph of $\wt Q(\omega)$ is in general not symmetric, i.e., it does not coincide with $Q(\omega)$ as defined in Definition \ref{O15}.

\begin{example}
    For $Q_{{3}}$ from \Cref{ex:i} we have ${\rm ran}\gamma\cap{\rm ran} A =\{0\}$ and hence $Q_{{3}}(0)$ is trivial, i.e. $Q_{{3}}(0)=\{(0;0)\}$,  whereas $\lim\limits_{\lambda\hat\to0} Q_3(\lambda)={{3}}i$.
\end{example}

\begin{example}
    For the matrix function $Q_5$ from Examples \ref{ex:N-null-matrix} and \ref{ex:zero-pole} it can be shown that $Q_5(0)=\Span\left\{\left(\binom 01;\binom 00
    \right), \left(\binom 00;\binom 10
    \right)\right\}$. Here $Q_5(0)$ has a nontrivial multi-valued part, which reflects the pole at $\omega=0$.
\end{example}

\subsection{A characterisation of generalised values}

In what follows we connect the notions from the preceding definitions. 

\begin{theorem}
\label{O33}
Let $Q \in \mc N(\mc G)$ have the minimal realisation $(\Pi,A,\gamma)$ and let $\omega \in \bb R$. For ${f_0},{g_0} \in \mc G$ the following statements are equivalent.
\begin{itemize}
    \item[{\rm(A)}] The function $Q$ has a directional function $f$ at $\omega$, such that $$
    f(\lambda){\to}f_0 \quad \text{ and }\quad Q(\lambda)f(\lambda){\to}g_0\quad \text{ as } \lambda\widehat{\to}\omega.
    $$
    \item[{\rm(A')}] $({f_0};{g_0})$ is a directional pair of $Q$ at $\omega$.
    \item[{\rm(B)}] $({f_0};{g_0}) \in Q(\omega)$.
    \item[{\rm(C)}] There exists an element ${v} \in \Pi$ such that 
    \begin{align}\label{eq:f}
     \gamma {f_0} &=\big(I+(\lambda_0-\omega)(A-\lambda_0)^{-1} \big){v}, \\\label{eq:g}
    {g_0} &= Q(\lambda_0)^* {f_0} +(\omega-\overline{\lambda_0}) \gamma^+ {v}.   
    \end{align}
\end{itemize}
The element ${v}$ is uniquely determined by $({f_0};{g_0})$.
If $Q$ is strict  then, conversely, $({f_0};{g_0})$ is uniquely determined by ${v}$.
\end{theorem}

The equivalence of (A) and (A') is part of \Cref{O31} (and is  repeated {here} only for the sake of completeness).
Implication (A),(A') $\Rightarrow$ (C) is  given by \Cref{O42}. The remaining implications will be shown below.

\begin{remark}
    In the case of generalised poles and zeroes, cf., Remark \ref{re:earlier}, the element $v\in\Pi$ from Theorem \ref{O33} can be identified as an eigenvector of the representing relation of the function $Q$ or $\widehat Q$, respectively. This was proven e.g., in \cite{luger:2006}, but is also re-obtained  when combining Theorem \ref{O33} with Lemma \ref{le:eigenvalue} {below}. 
\end{remark}

\subsubsection{The equivalence of (B) and (C)}\label{BC}

This step in the proof of \Cref{O33} is a consequence of  the computation rules for linear relations. For the convenience of the reader we provide the details.

\begin{proof} 
We have $({f_0};{g_0}) \in Q(\omega)$ if and only if
\begin{align*}
    \Big(\gamma {f_0};\frac{{g_0}-Q(\lambda_0)^*{f_0}}{\omega-\overline{\lambda_0}} \Big) \in \gamma^+ \big(I+(\omega-\lambda_0)(A-\omega)^{-1} \big).
\end{align*}
This is further equivalent to the existence of ${v} \in \Pi$ with
\begin{align}
\label{O66}
    &\Big({v};\frac{{g_0}-Q(\lambda_0)^*{f_0}}{\omega-\overline{\lambda_0}} \Big) \in \gamma^+, \\[.5ex]
\label{O65}
    &(\gamma {f_0};{v}) \in I+(\omega-\lambda_0)(A-\omega)^{-1}.
\end{align}
Now \eqref{O66} is equivalent to \eqref{eq:g}. Since (in the sense of linear relations)
\begin{align}
    I+(\omega-\lambda_0)(A-\omega)^{-1}=\big(I+(\lambda_0-\omega)(A-\lambda_0)^{-1} \big)^{-1}
\end{align} 
it also follows that \eqref{O65} is equivalent to \eqref{eq:f}. 
\end{proof}

\subsubsection{The implication (B),(C) $\Rightarrow$ (A),(A')}

This remaining implication is the crucial one since it requires the construction of a directional function.
Our strategy here, inspired by \cite{borogovac.luger:2014}, will be the following: We first handle a simple case and thereafter apply M\"obius transformations in order to be able to use the result from the simple case.

To start with, the following fact, which has been used without reference in the proof of \cite[Theorem~3.5]{borogovac.luger:2014}, seems to be folklore. Since \Cref{O43} below also depends on it, we shall include a proof.

\begin{lemma}
\label{O36}
    If $A$ is a self-adjoint linear relation in a Pontryagin space $\Pi$ and $\omega \in \bb R$ is not an eigenvalue of $A$ then 
    \begin{align*}
        \lim_{\lambda \hat{\to} \omega} (\lambda-\omega)(A-\lambda)^{-1}x=0
    \end{align*}
    for all $x \in \Pi$.
\end{lemma}
\begin{proof}
We assume without loss of generality that $A$ is an operator (as the multi-valued part of $A$ has no influence on the validity of the assertion).
Let $E$ be the local spectral function of $A$, and let  $K_0$ be the Riesz spectral subspace corresponding to the nonreal spectrum of $A$, cf. \cite{langer:1982}. Since $\omega$ is not an eigenvalue of $A$ it is not a critical point (this is due to \cite[Proposition~5.4]{langer:1982} and the fact that $A$ has a definitizing polynomial that is nonnegative on $\bb R$, cf. \cite[p.~12]{langer:1982}). Hence there exists an open interval $\Delta$ around $\omega$ such that $\ran E(\Delta)$ is a positive subspace of $\Pi$. Then $\Pi$ decomposes into the orthogonal direct sum
\begin{align*}
    \Pi=K_0 \, [\dotplus] \, \ran E(\bb R \setminus \Delta) \, [\dotplus] \, \ran E(\Delta),
\end{align*}
where each of the three components on the right-hand side is invariant under $A$. The point $\omega$ {belongs} to {n}either  $\sigma (A|_{K_0})$ {n}or $\sigma (A|_{\ran E(\bb R \setminus \Delta )})$, and thus the restrictions of $(\lambda-\omega)(A-\lambda)^{-1}$ to $K_0$ and $\ran E(\bb R \setminus \Delta )$ converge to zero (in norm) as $\lambda \hat{\to} \omega$. Since $\ran E(\Delta)$ is a Hilbert space, $(\lambda-\omega)(A-\lambda)^{-1}|_{\ran E(\Delta)}$ converges strongly as $\lambda \hat{\to} \omega$ to the orthogonal projection onto the eigenspace of $A|_{\ran E(\Delta)}$ at $\omega$ which, by assumption, is trivial.
\end{proof}

As in \cite{borogovac.luger:2014} the construction of a directional function is particularly simple for a point which is not  a generalised pole. In this case the following result gives a partial converse of \Cref{O42}, i.e., a special case of the implication (C) $\Rightarrow$ (A).

\begin{lemma}
\label{O43}
Let $Q \in \mc N(\mc G)$ have the minimal realisation $(\Pi,A,\gamma)$. Assume that $\omega \in \bb R$ is not an eigenvalue of $A$. If ${v} \in \Pi$ and ${f_0} \in \mc G$ are such that  
\begin{align*}
\gamma {f_0} =\big(I+(\lambda_0-\omega)(A-\lambda_0)^{-1} \big){v}
\end{align*}
then the constant function  $f(\lambda):={f_0}$ is a directional function of $Q$ at $\omega$ and
\begin{align*}
\lim_{\lambda \hat{\to} \omega} Q(\lambda){f_0} = Q(\lambda_0)^* {f_0} +(\omega-\overline{\lambda_0}) \gamma^+ {v}.
\end{align*}
\end{lemma}

\begin{proof}
We first note that the assumption on $v$ and $f_0$ implies
\begin{align*}
\gamma (\lambda) {f_0} = \big(I+(\lambda-\lambda_0)(A-\lambda)^{-1} \big)\gamma {f_0}=\big(I+(\lambda-\omega)(A-\lambda)^{-1} \big){v}
\end{align*}
for all $\lambda \in \rho (A)$. Since $\omega$ is not an eigenvalue of $A$, \Cref{O36} shows that $(\lambda - \omega) (A-\lambda)^{-1}$ converges strongly to zero as $\lambda \hat{\to} \omega$ and hence
\begin{align*}
&\gamma (\lambda) {f_0} \xrightarrow{\lambda \hat{\to} \omega} {v}. 
\end{align*}
With this we obtain
\begin{align*}
Q(\lambda){f_0}=Q(\lambda_0)^* {f_0}+(\lambda-\overline{\lambda_0})\gamma^+ \gamma (\lambda) {f_0} \xrightarrow{\lambda \hat{\to} \omega} Q(\lambda_0)^* {f_0}+(\omega-\overline{\lambda_0})\gamma^+ {v}.
\end{align*}
and
\begin{align*}
& \big(K_Q(\lambda,\lambda){f_0},{f_0} \big)_{\mc G}=[\gamma(\lambda){f_0},\gamma (\lambda){f_0}]_{\Pi}\xrightarrow{\lambda \hat{\to} \omega} [{v},{v}]_{\Pi}
\end{align*}
and thus $f(\lambda)\equiv f_0$ is a directional function of $Q$ at $\omega$ with generalised value given by \eqref{eq:g}.
\end{proof}
However, we know that generally speaking, constant directional functions do not suffice. Our strategy is now to apply a Möbius transformation to the function $Q$ such that the new function will satisfy the assumptions in \Cref{O43}, where we have to keep track of the transformations of all other objects, e.g., the corresponding directional pairs.

To start with we have a closer look at the crucial assumption in Lemma \ref{O43}. It means that $\omega\in\mathbb R$  is not a generalised pole of $Q$. Characterisations of generalised poles of various flavours can e.g., be found in \cite{borogovac.langer:1988,luger:2006,borogovac.luger:2014}. For the convenience of the reader we give a particular instance of such that fits our needs in this place. 

\begin{lemma}
\label{le:eigenvalue}
Let $Q \in \mc N(\mc G)$ have the minimal realisation $(\Pi,A,\gamma)$, and let $\omega \in \bb R$ and ${g_0} \in \mc G$. Then the following are equivalent:
\begin{itemize}
    \item[{\rm(i)}] $\omega$ is an eigenvalue of $A$ with  eigenvector ${v}$  such that ${g_0}=(\omega-\overline{\lambda_0})\gamma^+ {v}$.
\item[\rm{(ii)}] $(0;{g_0}) \in Q(\omega)$ and ${g_0} \neq 0$.

\end{itemize}
\end{lemma}

\begin{proof}
If  $(0;{g_0}) \in Q(\omega)$  then by the equivalence (B) $\Leftrightarrow$ (C) in Subsection \ref{BC}, there is an element ${v} \in \Pi$ such that
\begin{align}
\label{O68}
0&=\big(I+(\lambda_0-\omega)(A-\lambda_0)^{-1} \big){v}, \\
\label{O69}
{g_0} &=(\omega-\overline{\lambda_0}) \gamma^+ {v}.
\end{align}

If $g_0 \neq 0$, we infer from \eqref{O69} {that} ${v} \neq 0$, and \eqref{O68} shows that ${v}$ is an eigenvector of $A$.

Conversely, assume that ${v}$ is an eigenvalue of $A$. Defining ${g_0}$ by \eqref{O69}, again equivalence (B) $\Leftrightarrow$ (C)  yields $(0;{g_0}) \in Q(\omega)$. We have to show that ${g_0} \neq 0$. Assume the contrary, that is, $(\gamma^+{v},u)=0$ for all $u \in \mc G$. Since ${v}$ is an eigenvector of $A$ we have 
\[
\gamma(\lambda)^+ {v}=\gamma^+ \big(I+(\overline{\lambda}-\overline{\lambda_0})(A-\overline{\lambda})^{-1} \big){v}=\frac{\omega-\overline{\lambda_0}}{\omega-\overline{\lambda}}\gamma^+ {v}, \quad \lambda \in \rho (A).
\]
Hence 
\[
[{v},\gamma(\lambda)u ]_{\Pi}=(\gamma (\lambda)^+{v},u)_{\mathcal G}=0, \quad u \in \mc G, \, \lambda \in \rho (A).
\]
Using the minimality condition \eqref{O29} we find that ${v}=0$, contradicting our assumption that ${v}$ is an eigenvector of $A$.
\end{proof}

Next, observe how basic linear fractional transformations affect the corresponding directional functions.
\begin{lemma}
\label{O32}
Let $Q \in \mc N(\mc G)$ and $\omega \in \bb R$. Suppose that $f$ is a directional function of $Q$ at $\omega$ with directional pair $({f_0};{g_0})$. The following statements hold.
\begin{Enumerate}
\item If $\Theta \in \mc B(\mc G)$ is self-adjoint, then $f$ is also a directional function of $Q+\Theta$ at $\omega$ with directional pair $({f_0};\Theta {f_0}+{g_0})$.
\item If $Q$ is regular, then $\hat f(\lambda) \DE Q(\lambda)f(\lambda)$ is a directional function of $\wh Q$ at $\omega$ with directional pair $({g_0};-{f_0})$.
\end{Enumerate}
\end{lemma}
\begin{proof}
Item (i) is trivial and item (ii) follows from \eqref{O54}.
\end{proof}
Next we check that also the relation $Q(\omega)$ behaves as expected with respect to M\"obius transformations.  Note that this very natural statement is not at all obvious, since $Q(\omega)$ cannot be interpreted as a limit.

\begin{lemma}
\label{O75}
    Let $Q \in \mc N(\mc G)$  and $\omega \in \bb R$. Then the following statements hold.
    
    \begin{Enumerate}
        \item If $\Theta \in \mc B(\mc G)$ is a self-adjoint constant then
        \begin{align*}
           (Q+\Theta)(\omega)=Q(\omega)+\Theta,
        \end{align*}
        
        where $Q(\omega)+\Theta=\{({f_0};g_0+\Theta f_0) \DS ({f_0};{g_0}) \in Q(\omega) \}$ and  $(Q+\Theta)(\omega)$ is defined as in \Cref{O15} for the function $Q+\Theta$ $\in \mc N(\mc G)$.
        
        \item  If $Q$ is regular then
        \begin{align*}
        \wh Q(\omega)=-Q(\omega)^{-1},
    \end{align*}
    
    where $-Q(\omega)^{-1}= \{(-{g_0};{f_0}) \DS ({f_0};{g_0}) \in Q(\omega) \}$ and $\wh Q(\omega)$ is defined as in \Cref{O15} for the function $\wh Q\in \mc N(\mc G)$.
    \end{Enumerate}
\end{lemma}

\begin{remark}
    Later we will see that for the sum of two nonconstant functions statement (i) does not hold, i.e., in general $(Q_1+Q_2)(\omega)\neq Q_1(\omega)+Q_2(\omega)$. This issue will be discussed further in \Cref{S42}.
\end{remark}

\begin{proof}
The equality in (i) is obvious, since $Q$ and $Q+\Theta$ share a minimal realisation. We shall now prove (ii).
Choose $\lambda_0 \in \mf h (Q)$ such that $Q(\lambda_0)$ is boundedly invertible.
    We now make use of the characterisation in Subsection \ref{BC} and show that \eqref{eq:f}, \eqref{eq:g} are equivalent to the corresponding statements for $\wh Q$, i.e.,
    \begin{align}
    \label{O72}
        \hat{\gamma}(-{g_0}) &=\big(I+(\lambda_0-\omega)(\wh A-\lambda_0)^{-1} \big){v}, \\
        \label{O73}
        {f_0} &= \wh Q(\lambda_0)^* (-{g_0})+(\omega-\overline{\lambda_0})\hat{\gamma}^+ {v}.
    \end{align}
Due to symmetry (in the situation of \Cref{O1} we have $\hat{\hat{\gamma}}=-\gamma$ and $\wh{\wh {A \,}}=A$) one implication is sufficient. We choose to show that  \eqref{eq:f}, \eqref{eq:g} imply \eqref{O72}, \eqref{O73}.

Note that \eqref{O73} follows from \eqref{eq:g} by multiplication with $\wh Q(\lambda_0)^*$. Showing that \eqref{O72} follows from \eqref{eq:f} requires a little more work. As preparation we compute, using \eqref{eq:f},
\begin{align}
\nonumber
    (\lambda_0-\omega)\gamma (\overline{\lambda_0})^+{v} &=\gamma^+ \big[(\lambda_0-\omega)I+(\lambda_0-\overline{\lambda_0})(\lambda_0-\omega)(A-\lambda_0)^{-1} \big]{v} \\
    \nonumber
    &= \gamma^+ \big[(\lambda_0-\omega){v}+(\lambda_0-\overline{\lambda_0})(\gamma {f_0}-{v}) \big] \\
    \label{O74}
    &=\gamma^+\big[(\lambda_0-\overline{\lambda_0})\gamma {f_0}+(\overline{\lambda_0}-\omega){v} \big].
\end{align}
Recalling \eqref{O24} and applying \eqref{O74} we obtain
\begin{align*}
    &\big(I+(\lambda_0-\omega)(\wh A-\lambda_0)^{-1} \big){v}-\big(I+(\lambda_0-\omega)(A-\lambda_0)^{-1} \big){v} \\
    &=(\lambda_0-\omega)\gamma \wh Q(\lambda_0) \gamma(\overline{\lambda_0})^+{v} =\hat{\gamma} \gamma^+[(\lambda_0-\overline{\lambda_0})\gamma {f_0}+(\overline{\lambda_0}-\omega){v} ] \\
    &=\hat{\gamma} [(Q(\lambda_0)-Q(\lambda_0)^*){f_0}+(\overline{\lambda_0}-\omega) \gamma^+ {v} ] =-\gamma {f_0}+\hat{\gamma} (-{g_0}).
\end{align*}
Using again \eqref{eq:f} shows that \eqref{O73} is satisfied.
\end{proof}

Let us outline how \Cref{O32} and \Cref{O75} will be used in the proof of the implication (B),(C) $\Rightarrow$ (A),(A'). For some (still unspecified) self-adjoint $\Theta \in \mc B(\mc G)$ consider the function $\wh{Q+\Theta}$, which again belongs to $\mc N(\mc G)$. According to \Cref{O75}, a given $(f_0;g_0) \in Q(\omega)$ transforms into an element $(\tilde f_0,\tilde g_0)$ of $(\wh{Q+\Theta})(\omega)$. We want  to choose $\Theta$ so that we can apply \Cref{O43},  which then provides a (constant) directional function of $\wh{Q+\Theta}$ at $\omega$ with directional pair $(\tilde f_0,\tilde g_0)$. Then \Cref{O32}   will allow us to  transform back to obtain the desired directional function for $Q$ itself.

First, to fix notations
  choose for the function $Q\in \mc N(\mc G)$ a minimal realisation $(\Pi,A,\gamma)$.
 Then, for  self-adjoint $\Theta \in \mc B(\mc G)$, the perturbed function $Q+\Theta$ {has the same minimal realisation}. If $0 \in \rho(Q(\lambda_0)+\Theta)$ we define $\wh{A_\Theta}$ as in \eqref{O24}, i.e.,
\begin{align}
\label{O67}
(\wh{A_\Theta}-\lambda_0)^{-1} \DE (A-\lambda_0)^{-1}+\gamma(\lambda_0)\widehat{\left( {Q+\Theta}\right)}(\lambda_0)\gamma(\overline{\lambda_0})^+.
\end{align}
According to \Cref{O1}, $\wh{A_\Theta}$ is the representing relation in a minimal realisation of $\wh{Q+\Theta}$.

Now we are ready to show that to a given function $Q$ there can always  be added a self-adjoint constant such that the perturbed function is regular and its inverse satisfies the assumptions in Lemma \ref{O43}.

\begin{lemma}
\label{O25}
Let $Q \in \mc N ({\mc G})$ have the minimal realisation $(\Pi,A,\gamma)$, and let $\omega \in \bb R$. Then there exists a self-adjoint $\Theta \in \mc B({\mc G})$ such that $0 \in \rho \big(Q(\lambda_0)+\Theta \big) $ and $\omega$ is not an eigenvalue of $\wh{A_\Theta}$, defined by \eqref{O67}.
\end{lemma}

\begin{proof}
We consider $\Theta$ of the form $rI_{\mc G}+P$, with $r>\|Q(\lambda_0)\|+1$ and $P$ an orthogonal projection in $\mc G$. For such $\Theta$, invertibility of $Q(\lambda_0)+\Theta$ is clear as $\|Q(\lambda_0)+P\| \leq \|Q(\lambda_0)\|+1 <r$. Hence $\wh{A_\Theta}$ is well-defined for all such $\Theta$. We claim that $P$  can be chosen  such that also the second requirement is satisfied.

According to \Cref{le:eigenvalue} the number $\omega$ is an eigenvalue of $\wh{A_\Theta}$ if and only if the generalised value $\wh{(Q+\Theta)}(\omega)$ contains a nontrivial element of the form $(0;{g_0})$. By \Cref{O75} this is equivalent to $\ker (Q+\Theta)(\omega) \neq \{0\}$.
Let $S \DE \overline{Q(\omega)}+rI _{\mc G}$, which is a closed symmetric linear relation. It is enough to show that the orthogonal projection $P$ can be chosen such that $\ker (S+P)=\{0\}$, since 
\[
(Q+\Theta)(\omega)=Q(\omega)+rI_{\mc G}+P \subseteq  \overline{Q(\omega)}+rI_{\mc G}+P=S+P.
\]
The claim thus follows from the following statement:

\textit{Let $S$ be a closed symmetric linear relation in a Hilbert space and let $P$ be the orthogonal projection onto $\ker S$. Then $\ker (S+P)=\{0\}$.}

To see this, note that $x \in \ker (S+P)$ if and only if $(x;-Px) \in S$. This implies $Px \in \ran S \subseteq \ran S^* \subseteq (\ker S)^\perp$ and thus $Px=0$. Consequently $x \in \ker S$, which leads to $x=Px=0$.
\end{proof}

We are now ready to show the implication (B),(C) $\Rightarrow$ (A),(A'). 

\begin{proof}
Choose a minimal realisation $(\Pi,A,\gamma)$ of $Q$. For $\Theta=\Theta^* \in \mc B(\mc G)$ such that $0 \in \rho\big(Q(\lambda_0\big)+\Theta)$, let $(\Pi,\wh{A_\Theta},\hat{\gamma}_{\Theta})$ be the minimal realisation of $\wh{(Q+\Theta)}$ given by \eqref{O67}. By \Cref{O25} we can choose  $\Theta$ such that $\omega$ is not an eigenvalue of $\wh{A_\Theta}$ (in addition to $0 \in \rho(Q(\lambda_0)+\Theta)$). By \Cref{O75} we have
\begin{align*}
    ({f_0};{g_0}) \in Q(\omega) \quad &\Leftrightarrow \quad ({f_0};{g_0}+\Theta {f_0}) \in (Q+\Theta )(\omega) \\
    &\Leftrightarrow \quad (-{g_0}-\Theta {f_0};{f_0}) \in \wh{(Q+\Theta)}(\omega)
\end{align*}
By  (B) $\Leftrightarrow$ (C) this is further equivalent to the existence of an element $v\in\Pi$ with
\begin{align*}
\hat{\gamma}_{\Theta} (-{g_0}-\Theta {f_0}) =\big(I+(\lambda_0-\omega)(\wh{A_\Theta}-\lambda_0)^{-1} \big){v}, \\
{f_0} = \wh{(Q+\Theta )}(\lambda_0)^* (-{g_0}-\Theta{f_0}) +(\omega-\overline{\lambda_0}) (\hat{\gamma}_{\Theta})^+ {v}.
\end{align*}
Since $\omega$ is not an eigenvalue of $\wh{A_\Theta}$, \Cref{O43} implies that the {constant} function  $f_\Theta (\lambda) \DE-{g_0}-\Theta {f_0}$ is a directional function of $\wh{Q+\Theta}$ at $\omega$, and the corresponding directional pair is $(-{g_0}-\Theta {f_0};{f_0})$. By \Cref{O32} (ii) the function $f(\lambda) \DE \wh{(Q+\Theta})(\lambda) f_\Theta (\lambda)$
is a directional function of $Q+\Theta$ at $\omega$ with directional pair $({f_0};{g_0}+\Theta {f_0})$. Applying \Cref{O32} (i)  we finally find that the same function $f(\lambda)$ is a directional function of $Q$ with directional pair $({f_0};{g_0})$.  
\end{proof}

\subsubsection{The uniqueness statements}

Finally we need to prove the remaining statements in \Cref{O33}. 

\begin{proof}
Let first $({f_0};{g_0}) \in Q(\omega)$ and suppose that \eqref{eq:f}, \eqref{eq:g} are satisfied for two elements ${v}, {v}' \in \Pi$. It follows that
\begin{align}\label{eq:diff1}
0 =\big(I+(\lambda_0-\omega)(A-\lambda_0)^{-1} \big)({v}-{v}'), \\\label{eq:diff2}
0 = (\omega-\overline{\lambda_0}) \gamma^+ ({v}-{v}').
\end{align}
Hence if ${v}-{v}'$ is not zero, by \eqref{eq:diff1} it is an eigenvector of $A$. In this case  \Cref{le:eigenvalue} implies $(\omega-\overline{\lambda_0}) \gamma^+ ({v}-{v}') \neq 0$, contradicting \eqref{eq:diff2}. This shows ${v}={v}'$.

Now suppose in addition that $Q$ is strict, which is equivalent to $\ker \gamma =\{0\}$. Let ${v} \in \Pi$ and suppose that \eqref{eq:f}, \eqref{eq:g} are satisfied for two pairs $({f_0};{g_0}), \, (\tilde{{f_0}};\tilde{{g_0}}) \in Q(\omega)$. By \eqref{eq:f} we have $\gamma ({f_0}-\tilde{{f_0}})=0$ and hence ${f_0}=\tilde{{f_0}}$. Now \eqref{eq:g} shows that ${g_0}$ and $\tilde{{g_0}}$ coincide, too.
\end{proof}

\subsection{Interlude: the case $\kappa=0$}\label{sec-Nevanlinna}

This section deals with the special case of a classical Herglotz-Nevanlinna function, and hence a Hilbert space $\mc H$ in place of the Pontryagin space $\Pi$.
The reader, who is  only interested in either the main result or the indefinite situation, might skip this section.

Recall, if $Q$ belongs to $\mc N_0(\mc G)$ then $Q$ has a Herglotz-Nevanlinna integral representation
\begin{align}
\label{O96}
    Q(\lambda)=C+D\lambda+\int_{\bb R} \bigg(\frac{1}{t-\lambda}-\frac{t}{1+t^2}\bigg) \DD \Sigma (t), \qquad \lambda \in \bb C \setminus \bb R.
\end{align}
Here $C,D \in \mc B(\mc G)$ with $D \geq 0$, and $\Sigma$ is a $\mc B(\mc G)$-valued positive measure. 

This gives rise to an alternative representation of $Q(\omega)$, presented in \Cref{O35}. As we learned recently, a very similar representation can also be found in the unpublished manuscript \cite{hassi:spexit}.

Let us consider the multi-valued part of $Q(\omega)$ first.

\begin{lemma}
\label{O27}
    Let $Q \in \mc N_0(\mc G)$ and $\omega \in \bb R$. Then $\mul Q(\omega)=\ran \Sigma(\{\omega \})^{\frac 12}.$
\end{lemma}

\begin{proof}
    Let $(\mc H,A,\gamma)$ be a minimal realisation of $Q$, where $(\mc H,(.,.))$ is a Hilbert space. Then, for any $v \in \mc H$, it follows from the spectral theorem that $\lim_{\lambda \hat{\to} \omega} (\omega-\lambda)(A-\lambda)^{-1}v=P_{\omega}v$, where $P_{\omega}$ is the orthogonal projection onto $\ker(A-\omega)$. Hence, for any $h \in \mc G$,
    \begin{align}
    \label{O10}
        \lim_{\lambda \hat{\to} \omega} (\omega-\lambda) \gamma (\lambda)h=(\omega-\lambda_0)P_{\omega}\gamma h
    \end{align}
    and
    \begin{align}
    \label{O45}
    \begin{split}
        &\lim_{\lambda \hat{\to} \omega} (\omega-\lambda)Q(\lambda)h=\lim_{\lambda \hat{\to} \omega} (\omega-\lambda)\big(Q(\lambda_0)^*+(\lambda-\overline{\lambda_0})\gamma^*\gamma(\lambda) \big)h \\
        &=|\omega-\lambda_0|^2 \gamma^*P_{\omega} \gamma h.
    \end{split}
    \end{align}
    Using the integral representation \eqref{O96} of $Q$, it follows that
    \begin{align}
    \label{O12}
        \Sigma(\{\omega\})=\slim_{\lambda \hat{\to} \omega} (\omega-\lambda)Q(\lambda)=|\omega-\lambda_0|^2 \gamma^*P_{\omega} \gamma.
    \end{align}
    From \Cref{le:eigenvalue} we know that $\mul Q(\omega)=\ran (\gamma^* P_{\omega})$, and we claim that with \eqref{O12} this further equals $\ran (\gamma^* P_{\omega})=\ran \Sigma(\{\omega \})^{\frac 12}$. To prove this claim we use \cite[Corollary~2]{hassi:2004}, which gives  a characterisation of the range of a \ linear relation between Hilbert spaces. For a bounded operator this characterisation reads as follows:
    
    \emph{Let $T \in \mc B(\mc H_1,\mc H_2)$. Then $g \in \ran T$ if and only if there is a constant $C>0$ such that $|(h,g)| \leq C\|T^* h\|$ for all $h \in \mc H_2$.}

    With this statement in mind, let  $g \in \ran \gamma^*P_{\omega}$ with $v \in \mc H$ such that $g=\gamma^* P_{\omega} v$. Then for $h \in \mc G$ we have
    \begin{align*}
        |(h,g)_{\mc G}|&=|(P_{\omega}\gamma h,v)_{\mc H}| \leq \|P_{\omega}\gamma h \|_{\mc H} \cdot \|v\|_{\mc H}=\|v \|_{\mc H} \sqrt{(\gamma^* P_{\omega} \gamma h,h)_{\mc G}} \\
        &=\frac{\|v \|_{\mc H}}{|\omega-\lambda_0|} \sqrt{(\Sigma (\{\omega \})h,h)_{\mc G}}=\frac{\|v \|_{\mc H}}{|\omega-\lambda_0|} \|\Sigma (\{\omega \})^{\frac 12}h\|_{\mc G}.
    \end{align*}
   Hence $g \in \ran \Sigma (\{\omega \})^{\frac 12}$.

    Conversely, let $g \in \ran \Sigma (\{\omega \})^{\frac 12}$ with $f \in \mc G$ such that $g=\Sigma (\{\omega \})^{\frac 12}f$. If $h \in \mc G$ then
    \begin{align*}
        |(h,g&)_{\mc G}|=|(\Sigma (\{\omega \})^{\frac 12} h,f)_{\mc G}| \leq \|f \|_{\mc G} \sqrt{(\Sigma (\{\omega \}) h,h)_{\mc G}} \\
        &=|\omega-\lambda_0| \cdot \|f \|_{\mc G}  \sqrt{(\gamma^* P_{\omega} \gamma h,h)_{\mc G}}=|\omega-\lambda_0| \cdot \|f \|_{\mc G} \cdot \|(\gamma^* P_{\omega})^* h\|_{\mc H}.
    \end{align*}
    Using again the characterisation of the range of a bounded operator, it follows that $g \in \ran \gamma^*P_{\omega}$.
\end{proof}

Next up is a description of $\dom Q(\omega)$.

\begin{lemma}
\label{O26}
    Let $Q \in \mc N_0(\mc G)$ and $\omega \in \bb R$. Then
    \begin{align}
    \label{O23}
        \dom Q(\omega)=\bigg\{ f_0 \in \mc G \DS \int_{\bb R} \frac{\DD (\Sigma(t)f_0,f_0)}{(t-\omega)^2}<\infty \bigg\}.
    \end{align}
    Furthermore, for any $f_0 \in \dom Q(\omega)$ the constant function $f(\lambda) \DE f_0$ is a directional function of $Q$ at $\omega$.
\end{lemma}

\begin{proof}
    We start by showing that every $f_0 \in \dom Q(\omega)$ serves as a constant directional function for $Q$ at $\omega$.  Let a minimal realisation $(\mc H,A,\gamma)$ of $Q$ be fixed, then according to \Cref{O33} there exists $v \in \mc H$ such that \eqref{eq:f} holds true (we do not need \eqref{eq:g} here). It follows  that

    \begin{align*}
        \gamma (\lambda) f_0=(I+(\lambda-\omega)(A-\lambda)^{-1})v \xrightarrow{\lambda \hat{\to} \omega} (I-P_{\omega})v,
    \end{align*}
    where $P_{\omega}$ denotes the orthogonal projection onto $\ker (A-\omega)$. Using \eqref{O59} and \eqref{O93}, it follows that $f_0$ is a directional function of $Q$ at $\omega$. 
    This fact can be used to prove 
    the inclusion ``$\subseteq$'' in \eqref{O23}: By monotone convergence,
    \begin{align*}
        \int_{\bb R} \frac{\DD (\Sigma(t)f_0,f_0)}{(t-\omega)^2} &=\lim_{r \searrow 0} \int_{\bb R} \frac{\DD (\Sigma(t)f_0,f_0)}{|t-(\omega+ir)|^2}\\
        &=\lim_{r \searrow 0} \big(K_Q(\omega+ir,\omega+ir)f_0,f_0 \big)-(Df_0,f_0)<\infty.
    \end{align*}
    If, on the contrary, $\int_{\bb R} \frac{\DD (\Sigma(t)f_0,f_0)}{(t-\omega)^2}<\infty$, let us check that $f_0$ is a directional function of $Q$ at $\omega$. Let $W_{\omega,\theta} \DE \{{\omega} + re^{i\phi} \DS r>0, \, |\phi-\frac{\pi}{2}| \leq \frac{\pi}{2}-\theta \}$  for $\theta \in (0,\frac{\pi}{2})$ be an angular domain centered at $\omega$ and let $\lambda \in W_{\omega,\theta}$. Then 
    \begin{align*}
        \frac{1}{|t-\lambda|^2} \leq \frac{1}{(1-\cos \theta)|t-(\omega+ir)|^2} \leq \frac{1}{(1-\cos \theta)(t-\omega)^2}.
    \end{align*}
    Due to dominated convergence we have
    \begin{align*}
        \lim_{\lambda \in W_{\omega,\theta}} \big(K_Q(\lambda,\lambda)f_0,f_0 \big) &=(Df_0,f_0)+\lim_{\lambda \in W_{\omega,\theta}}\int_{\bb R} \frac{\DD (\Sigma(t)f_0,f_0)}{|t-\lambda|^2} \\
        &=(Df_0,f_0)+\int_{\bb R} \frac{\DD (\Sigma(t)f_0,f_0)}{(t-\omega)^2}<\infty.
    \end{align*}
    Hence $(\gamma(\lambda)f_0,\gamma(\lambda)f_0)$ converges as $\lambda \hat{\to} \omega$. Given the existence of the limit $\lim_{\lambda \hat{\to} \omega} (\gamma(\lambda)f_0,\gamma(\lambda)f_0)$, convergence of $\gamma(\lambda)f_0$ will be established if we can show that $\lim_{\lambda \hat{\to} \omega} (\gamma (\lambda)f_0,\xi)$ exists for all $\xi$ from a dense subset of $\mc G$. We use the dense subset $\Span \big\{\gamma (\mu) x \DS \mu \in \bb C \setminus \bb R, \, x \in {\mc G} \big\}$, and thus it will be enough to show that
    \begin{align*}
        (\gamma(\lambda)f_0,\gamma(\mu)x) &=(\gamma(\mu)^*\gamma(\lambda)f_0,x)=(K_Q(\lambda,\mu)f_0,x) \\
        &=(Df_0,x)+\int_{\bb R} \frac{\DD (\Sigma (t)f_0,x)}{(t-\lambda)(t-\overline{\mu})} 
    \end{align*}
    converges as $\lambda \hat{\to} \omega$. Due to $|(\Sigma(\cdot)f_0,x)|^2 \leq (\Sigma(\cdot)f_0,f_0)(\Sigma(\cdot)x,x)$ and Cauchy-Schwarz we get, for $\lambda \in W_{\omega,\theta}$, that
    \begin{align*}
        \int_{\bb R} &\frac{\DD |(\Sigma (t)f_0,x)|}{|(t-\lambda)(t-\overline{\mu})|} \leq \frac{1}{\sqrt{1-\cos \theta}}\int_{\bb R} \frac{\DD |(\Sigma (t)f_0,x)|}{|(t-\omega)(t-\overline{\mu})|} \\
        &\hspace{50pt}\leq \frac{1}{\sqrt{1-\cos \theta}}\bigg(\int_{\bb R} \frac{\DD (\Sigma (t)f_0,f_0)}{(t-\omega)^2} \int_{\bb R} \frac{\DD (\Sigma (t)x,x)}{|t-\mu|^2} \bigg)^{\frac 12}<\infty.
    \end{align*}
    Hence dominated convergence applies and  $(\gamma(\lambda)f_0,\gamma(\mu)x)$ converges.
\end{proof}

The lemmata \ref{O27} and \ref{O26} together give a complete description of the generalised value $Q(\omega)$. A similar result is given in the unpublished work \cite{hassi:spexit}.
\begin{theorem}
\label{O35}
    Let $Q \in \mc N_0(\mc G)$ and $\omega \in \bb R$. Then 
    \begin{align*}
        Q(\omega)=Q_{\emph{op}}(\omega) \dotplus (\{0\} \times \ran \Sigma (\{\omega \})^{\frac 12}).
    \end{align*}
    where
    \begin{align*}
        Q_{\emph{op}}(\omega) \DE \bigg\{ \big(f_0;\lim_{\lambda \hat{\to} \omega} Q(\lambda)f_0 \big) \DS \int_{\bb R} \frac{\DD (\Sigma(t)f_0,f_0)}{(t-\omega)^2}<\infty \bigg\}.
    \end{align*}
\end{theorem}
\begin{proof}
    Let $(f_0;g_0) \in Q(\omega)$. \Cref{O26} yields $\int_{\bb R} \frac{\DD (\Sigma(t)f_0,f_0)}{(t-\omega)^2}<\infty$ and existence of $g_0' \DE \lim_{\lambda \hat{\to} \omega} Q(\lambda)f_0$. Hence $(f_0;g_0') \in Q_{\text{op}}(\omega)$ and $(0;g_0-g_0') \in Q(\omega)$. By \Cref{O27} we have $(0;g_0-g_0') \in \{0\} \times \ran \Sigma(\{\omega\})^{\frac 12}$. Conversely, \Cref{O26} and \Cref{O27} give $Q_{\text{op}}(\omega) \subseteq Q(\omega)$ and $\{0\} \times \ran \Sigma (\{\omega \})^{\frac 12} \subseteq Q(\omega)$.
\end{proof}

So far in this section, we have described the generalised value, i.e., the relation $Q(\omega)$. The following results, however, concern directional functions.

First we construct particularly simple directional functions corresponding to elements in the range of $\Sigma(\{\omega \})$. Note that $\ran \Sigma(\{\omega \})$ may be a proper subset of $\mul Q(\omega)=\ran \Sigma(\{\omega \})^{\frac 12}$, in which case this construction does not yield a concrete directional function for all elements of $\mul Q(\omega)$.

\begin{proposition}
\label{O38}
    Let $Q \in \mc N_0(\mc G)$ and $\omega \in \bb R$. Pick $g_0 \in \ran \Sigma(\{\omega \})$ with $g_0=\Sigma(\{\omega \})h$. Then the function $f(\lambda) \DE (\omega-\lambda)h$ is a directional function of $Q$ at $\omega$ with directional pair $(0;g_0)$.
\end{proposition}
\begin{proof}
    By \eqref{O10},
    \begin{align*}
        \big((&K_Q(\lambda,\lambda)f(\lambda), f(\lambda)) \big)_{\mc G} =(\gamma (\lambda)f(\lambda),\gamma (\lambda) f(\lambda) )_{\mc H} \\
        &\xrightarrow{\lambda \hat{\to} \omega} |\omega-\lambda_0|^2 \| P_\omega \gamma h\|^2_{\mc H}.
    \end{align*}
    By \eqref{O45} and \eqref{O12} we have $\lim_{\lambda \hat{\to} \omega} Q(\lambda)f(\lambda)  \Sigma (\{\omega \})h=g_0$. The assertion follows since $\lim_{\lambda \hat{\to} \omega} f(\lambda) = 0$.
\end{proof}

In \Cref{O39} we claimed that the sum of two directional functions of $Q \in \mc N(\mc G)$ at $\omega \in \bb R$ is again a directional function. An important step in the proof is the following special case for $Q \in \mc N_0(\mc G)$.

\begin{lemma}
    \label{O71}
    Let $Q \in \mc N_0(\mc G)$ and $\omega \in \bb R$. If $f$ and $\tilde f$ are directional functions of $Q$ at $\omega$, then so is $f+\tilde f$.
\end{lemma}
\begin{proof}
    If $f$ and $\tilde f$ are directional functions then property  \eqref{eq:diff} obviously also holds for the sum. In order to see the second property \eqref{O37} we observe that by the positive definiteness of the kernel $K_Q$ we can use Cauchy-Schwarz, i.e., 
    \begin{align*}
        |\big(K_Q(\lambda,\lambda)f(\lambda),\tilde f(\lambda) \big)|^2 \leq \big(K_Q(\lambda,\lambda)f(\lambda),f(\lambda)\big) \big(K_Q(\lambda,\lambda)\tilde f(\lambda),\tilde f(\lambda)\big).
    \end{align*}
    Thus boundedness  of  $(K_Q(\lambda,\lambda)f(\lambda),f(\lambda))$ and $(K_Q(\lambda,\lambda) \tilde f(\lambda),\tilde f(\lambda))$  as $\lambda \hat{\to} \omega$, which holds by assumption,  implies also boundedness of the corresponding term for $f+\tilde f$.
\end{proof}

In the following result we further assume that the dimension of $\mc G$ is finite. This simplifies the  representation of $Q(\omega)$ from \Cref{O35}, since then the range of the matrix $\Sigma (\{\omega \})^{\frac 12}$ coincides with the range of $\Sigma (\{\omega \})$.

\begin{corollary}
\label{O97}
    Let $Q \in \mc N_0(\mc G)$, where $\dim \mc G<\infty$. Let further $\omega \in \bb R$. Then every $(f_0;g_0) \in Q(\omega)$ is obtained from some directional function that is a polynomial of degree not larger than one.
\end{corollary}
\begin{proof}
    We use \Cref{O35} to write $(f_0;g_0)=(f_0;Q_{\text{op}}(\omega)f_0)+(0;\tilde g_0)$, where 
    $\tilde g_0 \in \ran \Sigma (\{\omega \})^{\frac 12}$. According to \Cref{O26}, $\tilde f(\lambda) \DE f_0$ is a directional function of $Q$ at $\omega$, and by definition of $Q_{\text{op}}(\omega)$ it gives the directional pair $(f_0;Q_{\text{op}}(\omega)f_0)$.   
    Since $\dim \mc G<\infty$ we have 
    $\ran \Sigma (\{\omega \})^{\frac 12}=\ran \Sigma (\{\omega \})$, and 
    hence  \Cref{O38} implies that $(0;\tilde g_0)$
    is the directional pair corresponding to a directional function of the form $f(\lambda)=(\omega-\lambda)h$.
       By \Cref{O71}, the sum $f+\tilde f$ is a directional function with  directional pair $(f_0;g_0)$, and by construction it is a polynomial of degree at most one.
\end{proof}

The following example illustrates that the assumption $\dim \mathcal G <\infty$ is not only technical, but necessary.

\begin{example}\label{ex:ell2}
Consider the rational Nevanlinna function $Q(\lambda) \DE -\lambda^{-1} \Sigma_0 $ with some bounded nonnegative operator $\Sigma_0$ in a Hilbert space $\mathcal G$. Then the polynomial ansatz
$$
f(\lambda):= f_0+\lambda h_0
$$
with $f_0$, $h_0\in\mathcal G$ leads to the directional pair $(f_0; \Sigma_0 h_0)$, and the requirement $f_0\in\ker \Sigma_0$. 

This means that only directional pairs {whose} second component belongs to $\ran\Sigma_0$ can be obtained in this way. 
Let us now have a closer look at an even more concrete instance of this example. Let 
    \begin{align*}
        \Sigma_0 \FD{\ell^2(\bb N)}{\ell^2(\bb N)}{0}{(x_n)_{n \in \bb N}}{(\frac{x_n}{n})_{n \in \bb N}}
    \end{align*}
    and consider  $Q(\lambda) = -\lambda^{-1} \Sigma_0 \in \mc N_0(\ell^2(\bb N))$. By construction, the measure in the integral representation of $Q$ is the point measure with mass $\Sigma_0$ at $0$. Hence \Cref{O35} gives
 \begin{align*}
        Q(0)=\{0\} \times \ran \Sigma_0^{\frac 12} &=\{0\} \times \Big\{(y_n)_{n \in \bb N} \DF \sum_{n=1}^\infty n|y_n|^2<\infty \Big\}. 
    \end{align*}
    However,  we have seen above that a polynomial directional function at $0$ can only yield directional pairs of the form $(0;g_0)$, where
    \begin{align*}
        g_0 \in \ran \Sigma_0=\Big\{(y_n)_{n \in \bb N} \DF \sum_{n=1}^\infty n^2|y_n|^2<\infty \Big\} \subsetneq \ran \Sigma_0^{\frac 12}.
    \end{align*}
   Therefore, the requirement $\dim \mc G<\infty$ in \Cref{O97} is necessary. 
   Nonetheless, also for $(0;g_0)$ with $g_0=(g_{0,n})_{n \in \bb N} \in \ran \Sigma_0^{\frac 12} \setminus \ran \Sigma_0$ a directional function must exist. One such function is explicitly given by
    \begin{align*}
        f(\lambda) \DE \bigg(\frac{\lambda n }{1-\lambda n}g_{0,n} \bigg)_{n \in \bb N}.
    \end{align*}
    Note that  this function can be found by following the idea in the proof of \Cref{O33} (B) $\Rightarrow$ (A).
\end{example}

\section{Characterisation of eigenvalues of self-adjoint extensions with exit}
\label{S4}

This section contains our main result, which characterises the eigenvalues of the self-adjoint extension $\wt A$ given by Krein's formula in terms of generalised values of the $Q$-function $m$ and the parameter $\tau$.

We first formulate and prove the main theorem and discuss  then, with  help of several examples, its relation to known results as well as a kind of sharpness.

\subsection{The main result}
In this section we assume that  $S$ is a simple closed symmetric operator in the Pontryagin space $\Pi$, and  fix $A$,  a self-adjoint extension of $S$ in $\Pi$, with defect family $\gamma (\lambda)$   and $Q$-function  $m$. 

\begin{theorem}
\label{O55}
Assume that $(S,A,\gamma)$ and $m$ are given  as above and fix  some $\tau \in \mc N(\mc G)$ for which $m+\tau$ is regular.
 
Let $\wt A$ be a $\Pi$-minimal self-adjoint extension of $S$ in a Pontryagin space $\wt{\Pi}$ containing $\Pi$, such that
\begin{align}
\label{O60}
P_{\Pi}(\wt A-\lambda)^{-1}|_{\Pi}=(A-\lambda)^{-1}-\gamma (\lambda)(m(\lambda)+\tau(\lambda))^{-1}\gamma(\overline{\lambda})^+
\end{align}
holds for all $\lambda \in \rho (A) \cap \mf h((m+\tau)^{-1})$.

Then a point $\alpha \in \bb R$ is an eigenvalue of $\wt A$ if and only if there exist vectors $x_0,y_0\in\mathcal G$ with  $(x_0;y_0)\neq(0;0)$ such that 
\begin{equation}\label{eq:zeroes-main}
   (x_0;y_0)\in m(\alpha) \quad \text{ and } \quad(x_0;-y_0) \in \tau(\alpha). 
\end{equation}
Furthermore, there  exists a linear bijection between $\ker (\wt A-\alpha)$ and the space of all $(x_0;y_0) \in \mc G \oplus \mc G$ satisfying \eqref{eq:zeroes-main}.
\end{theorem}

In view of \Cref{O33} this result can  also equivalently be formulated in terms of directional functions. 
\begin{corollary}\label{cor:main-theorem}
    Let the relation $\wt A$ be given as an \Cref{O55}. Then a point $\alpha \in \bb R$ is an eigenvalue of $\wt A$ if and only if there exist directional functions $x_m$ of $m$ at $\alpha$ and $x_{\tau}$ of $\tau$ at $\alpha$, such that
\begin{align*}
    \lim_{\lambda \hat{\to} \alpha} x_m(\lambda) &=\lim_{\lambda \hat{\to} \alpha} x_{\tau}(\lambda), \\
    \lim_{\lambda \hat{\to} \alpha} m(\lambda)x_m(\lambda) &=-\lim_{\lambda \hat{\to} \alpha} \tau(\lambda) x_{\tau}(\lambda)
\end{align*}
and at least one of these limits is nonzero.
\end{corollary}

In the proof of \Cref{O55} we are going to use the same strategy as in \cite{behrndt.luger:2007}. Namely, we first identify the relation $\wt A$ as (unitarily equvalent to)  a minimal representing relation of an  $\mc N(\mc G \oplus \mc G)$-function, which is built from $m$ and $\tau$. Then the particular structure of this function makes it   possible to express its pole vectors in terms of directional pairs of $m$ and $\tau$ individually.

We start by introducing the function.

\begin{definition}
\label{O64}
Let $\tau,m \in \mc N(\mc G)$ and assume that $m+\tau$ and $\tau$ are both regular. We set
\begin{align}
\label{O2}
\begin{split}
\wt M &\DE - \begin{pmatrix}
m & -I_{\mc G} \\
-I_{\mc G} & \wh{\tau}
\end{pmatrix}^{-1} =\begin{pmatrix}
 -(m+\tau)^{-1} & (m+\tau)^{-1}\tau \\
\tau(m+\tau)^{-1} & m(m+\tau)^{-1}\tau
\end{pmatrix}.
\end{split}
\end{align}
\end{definition}

The following proposition is the motivation for this definition, namely allowing to identify $\wt A$ as a minimal representing relation for $\wt M$. For the case of finite defect it is formulated in \cite[Proposition~4.6]{behrndt.luger:2007}, but only minor adjustments are needed to make it work for infinite defect. Namely, in the end of the proof one has to consider the closure of the range of some operator instead of the range itself.

\begin{proposition}
\label{O57}
Let $(S,A,\gamma)$ be as above and let $m$ be a $Q$-function of $(S,A,\gamma)$. Suppose we are given $\tau \in \mc N(\mc G)$ such that $m+\tau$ and $\tau$ are both regular. Then 
\begin{enumerate}
\item The function $\wt M$ introduced in \Cref{O64} belongs to $\mc N(\mc G \oplus \mc G)$.
\item There exists a $\Pi$-minimal self-adjoint extension $\wh{\bb A}$ of $S$ with the following properties:
\begin{itemize}
\item[(i)] The equality
\begin{align}
\label{O7}
P_{\Pi} (\wh{\bb A}-\lambda)^{-1} \big|_{\Pi}=(A-\lambda)^{-1}-\gamma (\lambda) (m(\lambda)+\tau (\lambda))^{-1} \gamma (\overline{\lambda})^+
\end{align}
holds for all $\lambda \in \rho (A) \cap \mf h((m+\tau)^{-1})$.
\item[(ii)] $\wt M$ has a minimal realisation with representing relation $\wh{\bb A}$.
\end{itemize}
\end{enumerate}
\end{proposition}

Having provided the necessary tools, we can prove \Cref{O55}.

\begin{proof}[{Proof of \Cref{O55}}]
It is not a restriction to assume that $\tau$ is regular -- otherwise we can choose $\sigma \in \bb R$, $\sigma > \|\tau (\lambda_0)\|$ and replace $m$ by $m+\sigma I_{\mc G}$ and $\tau$ by $\tau-\sigma I_{\mc G}$. This keeps $m+\tau$ invariant and does not lead to any changes in Krein's formula, while \Cref{O75} (i) shows that the validity of the criterion in terms of $m(\alpha)$ and $\tau(\alpha)$ is unaffected. Hence, from now on we assume that $\tau$ is regular. Then the function $\wt M$ from \Cref{O64} and its minimal representing relation $\wh{\bb A}$ are available to us.

By comparing \eqref{O60} and \eqref{O7} we see that the compressed resolvents of $\wt A$ and $\wh{\bb A}$ coincide:
\begin{align}
\label{O77}
P_{\Pi}(\wt A-\lambda)^{-1}|_{\Pi}=P_{\Pi}(\wh{\bb A}-\lambda)^{-1}|_{\Pi}.
\end{align}
Both $\wt A$ and $\wh{\bb A}$ are $\Pi$-minimal extensions of $S$, which  together with \eqref{O77} implies that they are unitarily equivalent:
In the Hilbert space setting, this statement can be found in \cite[Theorem 4.2.3.]{behrndt.hassi.snoo:2020}. The same proof works in the Pontryagin space setting, because isometries with dense domain and dense range can still be continued to unitary operators \cite[Corollary 1 of Theorem 6.3]{iohvidov.krein.langer:1982}. Hence, $\wt A$ is also a minimal representing relation for $\wt M$. This enables us to apply   \Cref{le:eigenvalue}  to characterise its eigenvalues. Namely,  a point $\alpha \in \bb R$ is an eigenvalue of $\wt A$ if and only if $\wt M(\alpha)$ contains an element of the form $(0;{g_0})$ with ${g_0}=\binom{x_0}{y_0} \in \mc G \oplus \mc G \setminus \{0\}$. 

What remains is to rewrite this in terms of $m$ and $\tau$. Since this statement is of independent interest, we formulate it as a lemma.

\begin{lemma}
\label{O34}
Let $\tau,m \in \mc N(\mc G)$ and assume that $m+\tau$ and $\tau$ are regular. Define $\wt M$ by \eqref{O2} and let $\alpha \in \bb R$, $x_0,y_0 \in \mc G$. Then the following statements are equivalent:
\begin{Enumerate}
\item $\big( \binom 00;\binom{x_0}{y_0} \big) \in \wt M(\alpha)$.
\item $(x_0;y_0) \in m(\alpha)$ and $(x_0;-y_0) \in \tau (\alpha)$.
\end{Enumerate}
\end{lemma}

\begin{proof}
According to \Cref{O75} (ii),
\begin{align}
\label{O79}
    \bigg(\binom 00;\binom{x_0}{y_0} \bigg) \in \wt M(\alpha) \Leftrightarrow \bigg(\binom{x_0}{y_0}; \binom 00 \bigg) \in \wh{\wt M}(\alpha).
\end{align}
Here the function $\wh{\wt M}$ can be expressed as
\begin{align}
\label{O78}
\wh{\wt  M} &= \begin{pmatrix}
m & -I_{\mc G} \\
-I_{\mc G} & \wh{\tau}
\end{pmatrix}
=
\begin{pmatrix}
m & 0 \\
0 & \wh{\tau}
\end{pmatrix}
+
\begin{pmatrix}
0 & -I_{\mc G} \\
-I_{\mc G} & 0
\end{pmatrix}
.
\end{align}
Consider minimal realisations $(\Pi,A,\gamma)$ and $(\Upsilon,T,\gamma')$ of $m$ and $\tau$, respectively. Further let $(\Upsilon,\wh{T},\hat{\gamma}')$ be the minimal realisation of $\wh \tau$ given by \Cref{O1}. From \eqref{O78} we see that 
\[
\bigg(\Pi \times \Upsilon,
\begin{pmatrix}
    A & 0 \\
    0 & \wh T
\end{pmatrix}
,
\begin{pmatrix}
    \gamma & 0 \\
    0 & \hat{\gamma}'
\end{pmatrix} \bigg)
\]
is a minimal realisation of $\wh {\wt M}$. 

\Cref{O33} states that (the right-hand side of) \eqref{O79} is equivalent to the existence of a vector ${\bf v}=\binom{{v}}{{v} '} \in \Pi \times \Upsilon$ such that corresponding versions of \eqref{eq:f}, \eqref{eq:g} hold, that is,
\begin{align}
\begin{pmatrix}
    \gamma & 0 \\
    0 & \hat{\gamma}'
\end{pmatrix}
\binom{x_0}{y_0} 
&=
\bigg(I_{\Pi \times \Upsilon}+(\lambda_0-\alpha)\bigg(\begin{pmatrix}
    A & 0 \\
    0 & \wh T
\end{pmatrix}-\lambda_0 I_{\Pi \times \Upsilon} \bigg)^{-1} \bigg) \bf {v}, \\
\binom 00 &= \wh {\wt M}(\lambda_0)^* \binom{x_0}{y_0} +(\alpha-\overline{\lambda_0}) 
\begin{pmatrix}
    \gamma & 0 \\
    0 & \hat{\gamma}'
\end{pmatrix}^+ \bf {v}.
\end{align}
Writing out these equalities line by line we get 
\begin{align*}
\gamma x_0 &=\big(I_\Pi+(\lambda_0-\alpha)(A-\lambda_0)^{-1} \big){v}, \\
\hat{\gamma}' y_0 &=\big(I_\Upsilon+(\lambda_0-\alpha)(\wh T-\lambda_0)^{-1} \big){v}', \\
0 &= m(\lambda_0)^* x_0-y_0 +(\alpha-\overline{\lambda_0}) \gamma^+ {v}, \\
0 &= -x_0+\wh{\tau}(\lambda_0)^* y_0 +(\alpha-\overline{\lambda_0}) (\hat{\gamma}')^+ {v}'.
\end{align*}
Equivalently, we have
\begin{align}
\label{O80}
\begin{split}
\gamma x_0 &=\big(I_\Pi+(\lambda_0-\alpha)(A-\lambda_0)^{-1} \big){v}, \\
y_0 &= m(\lambda_0)^* x_0 +(\alpha-\overline{\lambda_0}) \gamma^+ {v},  
\end{split}
\\[1ex]
\label{O81}
\begin{split}
\hat{\gamma}' y_0 &=\big(I_\Upsilon+(\lambda_0-\alpha)(\wh T-\lambda_0)^{-1} \big){v}', \\
x_0 &= \wh{\tau}(\lambda_0)^* y_0 +(\alpha-\overline{\lambda_0}) (\hat{\gamma}')^+ {v}'.
\end{split}
\end{align}
Applying \Cref{O33} yields that \eqref{O80} is equivalent to $(x_0;y_0) \in m(\alpha)$. Analogously, \eqref{O81} is equivalent to $(y_0;x_0) \in \wh{\tau} (\alpha)$, and according to \Cref{O75} this holds if and only if $(x_0;-y_0) \in \tau (\alpha)$. 
\end{proof}

This finishes also the proof of the characterisation in \Cref{O55}. To prove the claimed one-to-one correspondence, in view of \Cref{O34} it is enough to find a linear bijection between $\ker (\wt A-\alpha)$ and the space of all $g_0=\binom{x_0}{y_0} \in \mc G \oplus \mc G$ for which $(0;g_0) \in \wt M(\alpha)$. With the notation from the proof of \Cref{O34}, define 
\begin{align}
    \Psi \FD{\ker (\wt A-\alpha)}{\mc G \oplus \mc G}{0}{{\bf v}}{(\alpha-\overline{\lambda_0})\begin{pmatrix}
        \gamma & 0 \\
        0 & \hat{\gamma}'
    \end{pmatrix}
    {\bf v}
    }
    .
\end{align}
According to \Cref{le:eigenvalue} the image of $\Psi$ equals $\{g_0 \DS (0;g_0) \in \wt M(\alpha)\}$. Injectivity of $\Psi$ follows from the uniqueness statement in \Cref{O33}.

\end{proof} 

Note that \Cref{O34} actually holds  more generally. To show this,  basically the same proof works. 
 \begin{corollary}
Let $\tau,m \in \mc N(\mc G)$ and assume that $m+\tau$ and $\tau$ are regular. Define $\wt M$ by \eqref{O2} and let $\alpha \in \bb R$, $a_0,b_0, x_0,y_0 \in \mc G$. Then the following statements are equivalent:
\begin{Enumerate}
\item $\big( \binom{a_0}{b_0};\binom{x_0}{y_0} \big) \in \wt M(\alpha)$.
\item $(x_0;y_0+a_0) \in m(\alpha)$ and $(x_0+b_0;-y_0) \in \tau (\alpha)$.
\end{Enumerate}
\end{corollary}

\subsection{Discussion of the main result}
\label{S42}

In \Cref{O55} eigenvalues of $\wt A$ are characterised through condition \eqref{eq:zeroes-main}. We can rephrase this in the following way: if $x_0 $ is fixed then \eqref{eq:zeroes-main} holds for some $y_0$ if and only if
\begin{align}
\label{eq:ker-sum-values}
    (x_0;0)\in m(\alpha) +\tau(\alpha). 
\end{align}

It might seem more natural to try to write this condition in terms of $(m+\tau)(\alpha)$, emphasizing the connection to Krein's formula \eqref{O60}. So a natural question is:
$$  
\text{Does } m(\alpha) + \tau(\alpha) \text{ coincide with } (m + \tau)(\alpha)\text{?}
$$
The answer is negative, in general, as the following simple example illustrates. 
 
\begin{example}
\label{O13}
 Let  $m(\lambda)=1-\frac1\lambda$   and $\tau(\lambda)=\frac1\lambda$. Here obviously for $\omega=0$ the element $(0;1)$ is a directional pair for both the $\mathcal N_0$-function $m$ and the $\mathcal N_1$-function $\tau$ and hence it holds 
 \begin{align*}
     (0;1)\in m(0)+\tau(0). 
 \end{align*}
However, since $m+\tau $ is the constant function $1$ we have
\begin{align*}
     (0;1)\not\in (m+\tau)(0). 
 \end{align*}
\end{example}

Let us have a closer look at the criterion \eqref{eq:zeroes-main} in \Cref{O55}. It splits  into two alternatives:

 \begin{itemize}
     \item If $x_0=0$, then $\alpha$ is a generalised pole of both $m$  and $\tau$ with the same pole vector $y_0$. 
     \item  If $x_0\neq0$, then $\ker\big(m(\alpha)+\tau(\alpha)\big)$ is nontrivial, cf. \eqref{eq:ker-sum-values}.
 \end{itemize}

Recall that in the scalar case ($\dim \mc G=1$)  the following criterion for a point $\alpha$ to be an eigenvalue of $\wt A$ was given  in \cite[Theorem~4.1]{behrndt.luger:2007}:  either $\alpha$ is a generalised pole of both $m$ and $\tau$ or $\alpha$ is a generalised zero of $m+\tau$. Note that this reflects the two alternatives above, where, notably, the second one  simplifies.

Another situation that allows for this kind of simplification is the Hilbert space setting, which will be discussed in the next subsection.

Aside from these two special cases, trying to simplify the criterion \eqref{eq:zeroes-main} is a delicate task, as cancellation as in \Cref{O13} is possible and  a point can be both a generalised zero and generalised pole. That is, both the kernel and the multi-valued part of the generalised value at this point  can be nontrivial.

\subsection{A simplification in the definite situation}\label{sec:def}

The eigenvalue criterion in \Cref{O55} can be given a more natural form if all involved operators act in Hilbert spaces. Note the similarity of the following result with \cite[Theorem~4.1]{behrndt.luger:2007}.

\begin{theorem}
\label{O52}
Let $S$ be a simple closed symmetric operator in the Hilbert space $\mc H$. Suppose that $A$ is a self-adjoint extension of $S$ in $\mc H$, that $\gamma (\lambda)$ is a defect family and that $m$ is a $Q$-function of $(S,A,\gamma)$. Furthermore, let $\tau \in \mc N_0(\mc G)$ be such that $m+\tau$ is regular.

Let $\wt A$ be an $\mc H$-minimal self-adjoint extension of $S$ in a Hilbert space $\wt{\mc H}$ containing $\mc H$, such that
\begin{align}
\label{O63}
P_{\mc H}(\wt A-\lambda)^{-1}|_{\mc H}=(A-\lambda)^{-1}-\gamma (\lambda)(m(\lambda)+\tau(\lambda))^{-1}\gamma(\overline{\lambda})^+
\end{align}
holds for all $\lambda \in \bb C \setminus \bb R$.

Then a point $\alpha \in \bb R$ is an eigenvalue of $\wt A$ if and only if it is either a generalised zero of $m+\tau$, or a generalised pole of both $m$ and $\tau$ with a common pole vector.

Furthermore, $\ker (m+\tau)(\alpha)$ is orthogonal to $\mul m(\alpha) \cap \mul \tau(\alpha)$, and there exists a linear bijection
\begin{align*}
    \wt \Psi : \, \ker (\wt A-\alpha) \to \ker (m+\tau)(\alpha) \oplus \big(\mul m(\alpha) \cap \mul \tau (\alpha) \big).
\end{align*}
\end{theorem}

\begin{remark}
    In the  case that $m$ and $\tau$ have a common pole vector there does not always exist  a common directional (pole cancellation) function, cf., \Cref{ex:IB}.
\end{remark}

The proof is a consequence of the following fact. Recall from the previous section that the asserted equality may fail if either $m$ or $\tau$ is not in $\mc N_0(\mc G)$.

\begin{lemma}
\label{O11}
    For $m,\tau \in \mc N_0(\mc G)$ and $\alpha \in \bb R$ we have
    \begin{align}
    \label{O51}
        (m+\tau)(\alpha)=m(\alpha)+\tau (\alpha).
    \end{align}
\end{lemma}

\begin{proof}
    Denote by $\Sigma_m$ and $\Sigma_\tau$ the measures in the integral representations of $m$ and $\tau$, respectively. Then the measure in the integral representation of $m+\tau$ is $\Sigma_m+\Sigma_\tau$. Thus \Cref{O26} yields $\dom (m+\tau)(\alpha)=\dom m(\alpha) \cap \dom \tau (\alpha)$ and further $(m+\tau)_{\text{op}}(\alpha)=m_{\text{op}}(\alpha)+\tau_{\text{op}}(\alpha)$. By \cite[Theorem~~2.2]{fillmore.williams:1971}, 
    \begin{align*}
        \ran \big(\Sigma_m (\{\alpha\})+\Sigma_\tau (\{\alpha\})\big)^{\frac 12}=\ran \Sigma_m (\{\alpha\})^{\frac 12}+\ran \Sigma_\tau (\{\alpha\})^{\frac 12}
    \end{align*}
    and the claim thus follows from \Cref{O35}.
\end{proof}

\begin{proof}[{Proof of \Cref{O52}}]
    We note first, using \Cref{O11}, that 
    \begin{align*}
        \ker (m+\tau)(\alpha) =\ker \big(m(\alpha)+\tau(\alpha) \big) \subseteq \dom m(\alpha) \subseteq \big(\mul m(\alpha)\big)^\perp \\
        \subseteq \big(\mul m(\alpha) \cap \mul \tau (\alpha)\big)^\perp.
    \end{align*}
    Hence $X \DE \ker \big((m+\tau)(\alpha)\big) \oplus \big(\mul m(\alpha) \cap \mul \tau (\alpha) \big)$ is indeed an orthogonal sum of subspaces of $\mc G$. We shall construct a linear bijection between   $X$ and $\ker (\wt A-\alpha)$. In view of \Cref{O55}, it is enough to find a linear bijection $\Lambda: \, X \to m(\alpha) \cap \big(-\tau (\alpha) \big)$. For $\eta_0 \in \mul m(\alpha) \cap \mul \tau (\alpha)$ we simply set $\Lambda \eta_0 \DE (0;\eta_0)$. For $x_0 \in \ker \big((m+\tau)(\alpha)\big)$, we need to find a deterministic way to pick an element $y_0 \in \mc G$ such that $(x_0;y_0) \in m(\alpha)$ and $(x_0;-y_0) \in \tau (\alpha)$. Since this requirement determines $y_0$ up to addition of an element of $V \DE \mul m(\alpha) \cap \mul \tau(\alpha)$, the definition
    \begin{align*}
        V_{x_0} \DE \Span \big((\{0\} \times V) \cup \{(x_0;y_0)\} \big)
    \end{align*}
    is independent of the choice of $y_0$. Clearly, the orthogonal complement of $\{0\} \times V$ in $V_{x_0}$ is one-dimensional and contains a unique element of the form $(x_0;\tilde y_0)$. We define $\Lambda x_0 \DE (x_0;\tilde y_0)$. One checks easily that $\Lambda$ is linear on $\ker \big((m+\tau)(\alpha)\big)$, and hence on all of $X$. Finally, $\Lambda$ is a bijection: For $(x_0;y_0) \in m(\alpha) \cap \big(-\tau (\alpha) \big)$ we have $\Lambda^{-1} (x_0;y_0)=x_0$ if $x_0 \neq 0$ and $y_0$ otherwise. 
\end{proof}

\subsection{Sharpness in the indefinite situation}
\label{S44}

In \Cref{S42} we asked: ``Does $m(\alpha) + \tau(\alpha)$ coincide with $(m + \tau)(\alpha)$?''. We saw that the answer is ``no'' when only assuming $m,\tau \in \mc N(\mc G)$ (\Cref{O13}), but ``yes'' if we assume $m,\tau \in \mc N_0(\mc G)$ (\Cref{O11}). The aim of this section is to further investigate the relationship between $m(\alpha) + \tau(\alpha)$ and $(m + \tau)(\alpha)$ as well as their directional functions. The following natural questions arise when investigating the eigenvalue criterion \eqref{eq:zeroes-main}:

\begin{enumerate}
    \item Does it hold that
    $$
    \ker\big(m(\alpha)+\tau(\alpha)\big)\subset \ker\big((m + \tau)(\alpha)\big)?
    $$
   I.e., do $(x_0;y_0)\in m(\alpha)$ and  $(x_0;-y_0)\in \tau(\alpha)$  imply  $(x_0;0)\in(m + \tau)(\alpha)$?   
    \item Conversely, does it hold that
    $$\ker\big(m(\alpha)+\tau(\alpha)\big)\supset \ker\big((m + \tau)(\alpha)\big)?
    $$
    I.e., does $(x_0;0)\in(m + \tau)(\alpha)$ imply the existence of some $y_0$ such that $(x_0;y_0)\in m(\alpha)$ and  $(x_0;-y_0)\in \tau(\alpha)$?
    \item  
    Is it possible in \Cref{cor:main-theorem} to choose the directional functions for $m$ and $\tau$ to be the same, i.e. $x_m(\lambda)=x_\tau(\lambda)$?
\end{enumerate}

The answers to all these questions will turn out to be negative. Hence it seems that \Cref{O55} in its general form is sharp and cannot be simplified without losing information.

The following example answers question 1 negatively.

\begin{example}
  Consider the $\mathcal N_1(\mathbb C^2)$-functions
\begin{align*}
    m(\lambda)=
    \begin{pmatrix}
        \frac1\lambda  & 1 \\
        1 & 2\lambda
    \end{pmatrix}
    \quad \text{ and }
    \quad
    \tau (\lambda)=
    \begin{pmatrix}
        -\frac1\lambda  & 1 \\
        1 & -2\lambda
    \end{pmatrix}.
\end{align*}  

Then $\binom{ -\lambda}{1}$ is a directional function for $m$ at $\omega=0$ and $\binom{ \lambda}{1}$  for $\tau$. This implies that 
\begin{align*}
  \bigg(\binom{ 0}{1} ; \binom{ 0}{0}\bigg)\in m(0) \cap\tau(0), 
\end{align*}
but 
\begin{align*}
    (m+\tau)(\lambda)=
    \begin{pmatrix}
        0  & 2 \\
        2 & 0
    \end{pmatrix}
\end{align*}
is an invertible constant. 
\end{example}
Here we had an example where both $m$ and $\tau$ even have a generalised zero with common root vector $\binom{ 0}{1}$, however, the sum $m+\tau$ does not have a generalised zero.

The next example will quickly give  a
negative answer to question 2.

\begin{example}
Consider the generalised Nevanlinna functions
\begin{align*}
    m(\lambda)=\frac1\lambda
    ,
    \qquad
    \tau (\lambda)= -\frac1\lambda+\lambda.
\end{align*}
Then we have $m(0)=\tau(0)=\Span\{(0;1)\}$ but  $(m+\tau)(0)=\Span\{(1;0)\}$. Hence  $\ker\big( m(0)+\tau(0)\big) $ is trivial and cannot contain  $\ker (m+\tau)(0)$, which is nontrivial.
\end{example}

We turn now to a slightly more complicated example, which shows even more. Here  $\alpha=0$ is a generalised zero of $m+\tau$, with nontrivial root vector $x_0$ which even belongs to the domains of both $m(0)$ and $\tau(0)$. Nonetheless, there does not exist $y_0$ such that $\big(x_0 ;y_0 \big) \in m(0)$ and $\big(x_0 ;-y_0 \big) \in \tau(0)$. 

\begin{example}\label{ex:question2}
Consider the generalised Nevanlinna functions
\begin{align*}
    m(\lambda)=
    \begin{pmatrix}
        \lambda^{-2}-\lambda^{-1} & \lambda^{-1} \\
        \lambda^{-1} & 1
    \end{pmatrix}
    ,
    \qquad
    \tau (\lambda)=
    \begin{pmatrix}
        \lambda^{-4}+1 & \lambda^{-2} \\
        \lambda^{-2} & 0
    \end{pmatrix}.
\end{align*}
Then we obtain
\begin{align}
     m(0)=\Span \bigg\{\bigg(\binom 01 ; \binom{0}{0} \bigg),\bigg(\binom 00 ; \binom 10 \bigg) \bigg\}, \\
     \tau(0)=\Span \bigg\{\bigg(\binom 01 ; \binom{0}{-1} \bigg),\bigg(\binom 00 ; \binom 10 \bigg) \bigg\}
\end{align}
as both relations are symmetric, and hence at most two-dimensional and we have directional functions   $\binom{-\lambda}{1 -\lambda}, \binom{0}{\lambda}$ of  $m$ and $\binom{-\lambda^2}{1}, \binom{0}{\lambda^2}$ of $\tau$ at $0$. Hence
\begin{align*}
    m(0)+\tau (0) = \Span \bigg\{\bigg(\binom 01 ; \binom{0}{-1} \bigg),\bigg(\binom 00 ; \binom 10 \bigg) \bigg\}.
\end{align*}
On the other hand 
\begin{align*}
    (m+\tau)(\lambda)=
    \begin{pmatrix}
        \lambda^{-4}+\lambda^{-2}-\lambda^{-1}+1 & \lambda^{-2}+\lambda^{-1} \\
        \lambda^{-2}+\lambda^{-1} & 1
    \end{pmatrix}
    .
\end{align*}
The functions  $\binom{0}{\lambda^2}, \binom{-\lambda^2}{1-\lambda+2\lambda^2}$ are directional functions of $m+\tau$ at $0$ and hence 
\begin{align*}
    (m+\tau)(0) = \Span \bigg\{\bigg(\binom 01 ; \binom 00 \bigg),\bigg(\binom 00 ; \binom 10 \bigg) \bigg\}.
\end{align*}
\end{example}

Regarding question 3 we give two examples. In the first one the functions are scalar but one function is a proper generalised Nevanlinna function, while in the second both functions are Herglotz-Nevanlinna functions but operator-valued.

\begin{example}
     Let $m(\lambda)=\lambda^{-1}$ and $\tau(\lambda)=\lambda^{-2}$. Clearly $(0;1)$ is a directional pair for both $m$ and $\tau$ at $0$, i.e., \eqref{eq:zeroes-main} is satisfied with $x_0=0$ and $y_0=1$. Yet $m$ and $\tau$ share no nontrivial directional function. Indeed, if $x(\lambda)$ is such a directional function of $\tau$ at $0$ then $x(\lambda)/\lambda^2$ converges to a nonzero limit as $\lambda \hat{\to} 0$. Hence $x(\lambda)$, although it certainly is a directional function for $m$, can only give the trivial directional pair. We see that the conditions stated in \Cref{cor:main-theorem} cannot be satisfied nontrivially with a common directional function.
\end{example}

\begin{example}\label{ex:IB}
    Let $\mc G$ be an infinite-dimensional Hilbert space and $B$  a bounded uniformly positive operator on $\mc G$ without eigenvalues. Set
    $$
     m(\lambda) \DE -\lambda^{-1} I_{\mc G}\qquad \text{ and }\qquad\tau (\lambda) \DE -\lambda^{-1} B=Bm(\lambda).
    $$
    Since $\ran B^{\frac 12}=\ran B=\mc G$, \Cref{O35} gives $m(0)=\tau(0)=\{0\} \times \mc G$. Hence \eqref{eq:zeroes-main} is satisfied for any tuple of the form $(0;y_0)$ with $y_0 \in \mc G$. Yet, any two directional functions $x_m$ and $x_\tau$ satisfying the criterion in \Cref{cor:main-theorem} must be linearly independent. Indeed, let us assume $x_m = c x_\tau$ with $\lim_{\lambda \hat{\to} \alpha}  m(\lambda)x_m(\lambda)=y_0\neq 0$ and some nontrivial constant $c$. This would lead to  
    \begin{align*}
        By_0&=\lim_{\lambda \hat{\to} \alpha} B m(\lambda) x_m(\lambda) =\lim_{\lambda \hat{\to} \alpha} \tau(\lambda) x_m(\lambda)=c\lim_{\lambda \hat{\to} \alpha} \tau(\lambda) x_\tau(\lambda) \\
        &=-c\lim_{\lambda \hat{\to} \alpha} m(\lambda) x_m(\lambda)=-c y_0.
    \end{align*}
    However, this is impossible due to the assumption that $B$ has no eigenvalues.
\end{example}

The above counterexamples for questions 1 and 2 share one property, namely that the point of interest  is  also  a generalised pole not of positive type for both $m$ and $\tau$ with a common pole vector. This allows for cancellations, which seem to be necessary for any counterexample.

However, if $m$ and $\tau$ do not share a pole vector at a point $\alpha$  then examples suggest that these cancellation effects cannot occur. This leads us to the following conjecture.

\begin{conjecture}
Under the assumptions of \Cref{O55} it holds $\alpha\in\sigma_p(\wt A)$ if and only if either there exists a nontrivial element $y_0$ such that \begin{align}\label{eq:conj1}
 (0;y_0)\in m(\alpha)\cap\tau(\alpha)   
\end{align} 
or there exists a nontrivial element $x_0$ such that 
\begin{align}\label{eq:conj2}
(x_0;0)\in (m+\tau)(\alpha). 
\end{align}
\end {conjecture}
Let us first mention that  indeed \cref{eq:conj1} and \cref{eq:conj2} are sufficient conditions for $\alpha\in\sigma_p(\wt A)$. For \cref{eq:conj1} this is part of \Cref{O55}. Condition \cref{eq:conj2} means that $\alpha$ is a generalised pole of the function $-(m+\tau)^{-1}$. Since $\wt A$ is a, not necessarily minimally, representing relation of this function, \cref{eq:conj2} implies that $\wt A$ has the  eigenvalue $\alpha$.

Moreover, we have seen that the conjecture holds in certain cases, namely,   if $m$ and $\tau$ are scalar \cite{behrndt.luger:2007} or both in $\mc N_0(\mc G)$ (cf. \Cref{O52}). More generally, we mention \Cref{O14} which implies that the conjecture holds if $\alpha$ is not a generalised pole of negative type of either $m$ or $\tau$.

Note that even if the conjecture was true, it cannot replace \Cref{O55}, as we know that in general there is no one-to-one correspondence between the eigenspace of $\wt A$ and the linear span of all elements $(0;y_0)$,  $(x_0;0)$ satisfying \eqref{eq:conj1} and \eqref{eq:conj2}, respectively. This is illustrated by \Cref{ex:question2}, where we have a generalised zero of $m+\tau$ that does not give rise to an eigenvector of $\wt A$. Still, the point $\omega=0$ is an eigenvalue of $\wt A$, but the reason is different, namely because $m$ and $\tau$ have the common pole vector $(1,0)^\top$.

\section{Local results}
\label{S6}

The central tool introduced in this paper is the generalised value $Q(\omega)$ of a function $Q \in \mc N(\mc G)$. In \Cref{O33} we gave a description of $Q(\omega)$ purely in terms of the values of $Q(\lambda)$ in a neighborhood of $\omega$. Therefore, it is not surprising that our results can be lifted to more general classes of functions, which behave only locally like functions in $\mc N(\mc G)$. Basically the results presented so far still hold, when the assumptions hold only locally. The purpose of this last section is to make this precise.

After recalling basic facts on operator representations of local generalised Nevanlinna functions, we will give a characterisation of generalised values analogous
to \Cref{O33}. Finally, in \Cref{OS63} we extend some of our results in the definite case also to functions that are only locally in $\mc N_0$.

For the purpose of defining local generalised Nevanlinna functions, we are often going to use open domains $\Omega \subseteq \bb C$ with the following additional properties:
\begin{align}
    \label{O50}
    \lambda \in \Omega \Leftrightarrow \overline{\lambda} \in \Omega, \qquad \Omega \cap \bb C^\pm \text{ is simply connected}, \qquad \Omega \cap \bb R \neq \emptyset.
\end{align}

\begin{remark}
    One might consider domains in the closed complex plane $\mathbb C\cup\{\infty\}$, but we refrain  from this technicality. 
\end{remark}

\begin{definition}
\label{O94}
    Let $\Omega$ be a domain satisfying \eqref{O50} and let $\mc G$ be a Hilbert space. Let $Q$ be a $\mc B(\mc G)$-valued meromorphic function on $\Omega$ that satisfies $Q(\overline{\lambda})=Q (\lambda)^*$ for $\lambda \in \Omega$. Then $Q$ is called a \emph{local generalised Nevanlinna function} if, for every $\Omega'$ satisfying \eqref{O50} and $\overline{\Omega'} \subseteq \Omega$, one can write $Q$ in the form $Q=Q_{\mc N}+H$ where $Q_{\mc N} \in \mc N(\mc G)$ and $H$ is holomorphic on $\Omega'$.

    The class of all $\mc B(\mc G)$-valued local generalised Nevanlinna functions in $\Omega$ will be denoted by $\mc N(\mc G,\Omega)$. 
\end{definition}

\subsection{Operator representations of functions in $\mc N(\mc G,\Omega)$}
\label{S61}

Local generalised Nevanlinna functions admit operator representations similar to generalised Nevanlinna functions. A major difference compared with the representation of a function in $\mc N(\mc G)$ is that the representing relation of $Q \in \mc N(\mc G,\Omega)$ may act in a proper Krein space. The following definition describes the kind of relations that can appear in an operator representation of a function in $\mc N(\mc G,\Omega)$.

\begin{definition}
\label{O82}
    Let $\Lambda$ be a domain satisfying \eqref{O50}, and let $T$ be a self-adjoint linear relation in a Krein space $(\mc K,[.,.])$. We say that $T$ is \emph{of type $\pi_+$ over $\Lambda$} if, for every domain $\Lambda'$ satisfying \eqref{O50} and $\overline{\Lambda'} \subseteq \Lambda$, there exists a self-adjoint projection $E$ in $\mc K$ such that $T$ can be decomposed as
    \begin{align}
        T=(T \cap (E\mc K)^2) [\dotplus] (T \cap ((I-E)\mc K)^2)
    \end{align}
and the following holds:
\begin{Enumerate}
    \item $(E\mc K,[.,.])$ is a Pontryagin space and the set $\rho(T \cap (E\mc K)^2)$ is
nonempty,
    \item $\sigma (T \cap ((I-E)\mc K)^2) \cap \Lambda'=\emptyset$.
\end{Enumerate}
\end{definition}

For functions in $\mc N(\mc G,\Omega)$ the question about representations of the form \eqref{O28} is a bit more delicate than in the case of $\mc N(\mc G)$, since the domain $\Omega$ also plays a role. 

\begin{definition}
\label{O83}
    Let $\Omega, \Lambda$ be domains satisfying \eqref{O50} with $\overline{\Lambda} \subseteq \Omega$. Suppose that $Q \in \mc N(\mc G,\Omega)$ is given on $\Lambda$ by
    \begin{align*}
Q(\lambda) &=Q(\lambda_0)^* \\
&+(\lambda-\overline{\lambda_0})\gamma^+ \big(I+(\lambda-\lambda_0)(T-\lambda)^{-1} \big)\gamma, \quad \lambda \in \mf h(Q) \cap \rho(T) \cap \Lambda,
\end{align*}
where $T$ is a self-adjoint linear relation in a Krein space $\mc K$ that is of type $\pi_+$ over $\Lambda$, $\lambda_0 \in \mf h(Q) \cap \rho(T) \cap \Lambda \cap \bb C^+$ and $\gamma: \mc G \to \mc K$ is a bounded linear map. Then we say that $(\mc K,T,\gamma)$ is a \emph{$\pi_+$-realisation of $Q$ over $\Lambda$}. If, in addition,
\begin{align}
\label{O86}
\mc K= \overline{\Span} \big\{{\big(I+(\lambda-\lambda_0)(T-\lambda)^{-1} \big)\gamma} x \DS \lambda \in \rho(T) \cap \Lambda, \, x \in {\mc G} \big\}
\end{align}
then we say that the realisation $(\mc K,T,\gamma)$ is \emph{minimal}.
\end{definition}

As in the Pontryagin space case we shall use the abbreviation
\begin{align*}
    \gamma(\lambda) \DE \big(I+(\lambda-\lambda_0)(T-\lambda)^{-1} \big)\gamma, \quad \lambda \in \rho(T).
\end{align*}

The aforementioned analogue of \Cref{prop:representation} reads as follows. The statement appeared first in \cite{behrndt.luger:2007} but in principle it is already contained (with proof) in \cite{jonas:2005}.

\begin{proposition}
\label{O85}
    Let $\Omega$ be as above. A function $Q$ meromorphic in $\Omega \setminus \bb R$ with values in $\mathcal B(\mathcal  G)$ belongs to $\mc N(\mc G,\Omega)$ if and only if, for every $\Lambda$ satisfying \eqref{O50} and $\overline{\Lambda} \subseteq \Omega$, there exists a minimal $\pi_+$-realisation of $Q$ over $\Lambda$. 
\end{proposition}

Like for functions in $\mc N(\mc G)$, we say that $Q \in \mc N(\mc G,\Omega)$ is regular if there exists $\lambda \in \mathfrak{h}(Q)$ such that $Q(\lambda)$ is invertible in $\mc B(\mc G)$. In this case we write $\wh Q(\lambda) \DE -Q(\lambda)^{-1}$. According to \cite[Proposition~2.6]{behrndt.luger:2007} the function $\wh Q$ belongs to $\mc N(\mc G,\Omega)$ as well. Moreover, if $Q$ has a minimal $\pi_+$-realisation $(\mc K,T,\gamma)$ over $\Lambda$ then $\wh Q$ has the minimal $\pi_+$-realisation $(\mc K,\wh T,\hat{\gamma})$ over $\Lambda$, where
\begin{align}
\label{O92}
        (\wh{T}-\lambda_0)^{-1} \DE (T-\lambda_0)^{-1}+\gamma(\lambda_0)\widehat{Q}(\lambda_0)\gamma(\overline{\lambda_0})^+
    \end{align}
    and $\hat{\gamma} \DE \gamma \wh Q(\lambda_0)$.

\subsection{The generalised value $Q(\omega)$, revisited}
\label{S62}

For $Q \in \mc N(\mc G,\Omega)$ and a real point $\omega$ we introduce directional functions in exactly the same way as in \Cref{O31}. Also the notions \emph{directional pair and root/pole function/vector} carry over in a verbatim way. For the definition of $Q(\omega)$ we first have to make a choice of a minimal realisation.

\begin{definition}
\label{O87}
Let $Q \in \mc N(\mc G,\Omega)$, choose $\omega \in \Omega \cap \bb R$, and fix a minimal $\pi_+$-realisation $(\mc K,T,\gamma)$ of $Q$ over a domain $\Lambda$ with $\omega \in \Lambda \subseteq \overline{\Lambda} \subseteq \Omega$. We associate to the function $Q$ the linear relation $Q(\omega)\in\mc G\times \mc G$ defined by
\begin{align*}
Q(\omega) \DE Q(\lambda_0)^*+(\omega-\overline{\lambda_0})\gamma^+\big(I+(\omega-\lambda_0)(T-\omega)^{-1} \big) \gamma.
\end{align*}
\end{definition}
We are going to show the following analogue for \Cref{O33}.

\begin{theorem}
\label{O99}
Suppose we are in the setting of \Cref{O87}. Then  the relation $Q(\omega)$ is independent of the choice of minimal realisation. For any fixed minimal $\pi_+$-realisation as in \Cref{O87} and for ${f_0},{g_0} \in \mc G$ the following statements are equivalent.
\begin{itemize}
    \item[{\rm(A)}] The vector ${g_0}$ is a generalised value of $Q$ at $\omega$ in the direction ${f_0}$.
    \item[{\rm(A')}] $({f_0};{g_0})$ is a directional pair of $Q$ at $\omega$.
    \item[{\rm(B)}] $({f_0};{g_0}) \in Q(\omega)$.
    \item[{\rm(C)}] There exists an element ${v} \in \mc K$ such that 
    \begin{align}\label{eq:fS6}
     \gamma {f_0} &=\big(I+(\lambda_0-\omega)(T-\lambda_0)^{-1} \big){v}, \\\label{eq:gS6}
    {g_0} &= Q(\lambda_0)^* {f_0} +(\omega-\overline{\lambda_0}) \gamma^+ {v}.   
    \end{align}
\end{itemize}
The element ${v}$ is uniquely determined by $({f_0};{g_0})$.
If $Q$ is strict, i.e., if \eqref{O41} holds for some $\mu \in \rho (T) \cap \Lambda$, then also $({f_0};{g_0})$ is uniquely determined by ${v}$.
\end{theorem}

The auxiliary results used in the proof of \Cref{O33} all do have analogues for local generalised Nevanlinna functions, with mostly identical proofs. An argument is needed for generalizing \Cref{O42} and \Cref{O43}, and we shall provide the necessary details in the proof of the following proposition.

\begin{proposition}
\label{O84}
Let $Q \in \mc N(\mc G,\Omega)$ and let $(\mc K,T,\gamma)$ be a minimal $\pi_+$-realisation of $Q$ over a suitable subdomain $\Lambda$ of $\Omega$. Let $\omega \in \Lambda \cap \bb R$, and suppose that $({f_0};{g_0}) \in \mc G \oplus \mc G$ is a directional pair of $Q$ at $\omega$. Then there exists a vector ${v} \in \mc K$ such that 
\begin{align}
\label{O88}
\gamma {f_0} =\big(I+(\lambda_0-\omega)(T-\lambda_0)^{-1} \big){v}, \\
\label{O89}
{g_0} = Q(\lambda_0)^* {f_0} +(\omega-\overline{\lambda_0}) \gamma^+ {v}.
\end{align}
Conversely, if $\omega$ is not an eigenvalue of $T$, and $v \in \mc K$ and $f_0,g_0 \in \mc G$ are related by \eqref{O88}, then the constant function $f(\lambda) \DE{f_0}$ is a directional function of $Q$ at $\omega$ and $g_0 = \lim_{\lambda \hat{\to} \omega} Q(\lambda){f_0}$.
\end{proposition}
\begin{proof}
By \Cref{O82} there is a subdomain $\Lambda'$ of $\Lambda$ and a self-adjoint projection $E$ such that the orthogonal decomposition $\mc K=E\mc K [\dotplus ] (I-E)\mc K$ splits $T$ into the orthogonal sum of a self-adjoint relation in the Pontryagin space $E\mc K$ and a self-adjoint relation with no spectrum in $\Lambda'$. Then, since $E$ commutes with $(T-\lambda)^{-1}$,
\begin{align}
\label{O90}
    &Q(\lambda) =\underbrace{Q(\lambda_0)^*+(\lambda-\overline{\lambda_0})(E\gamma)^+ \big(I+(\lambda-\lambda_0)(T-\lambda)^{-1} \big)E\gamma}_{\ED Q_{\mc N}(\lambda)} \\
    \nonumber
    &+\underbrace{(\lambda-\overline{\lambda_0})((I-E)\gamma)^+ \big(I+(\lambda-\lambda_0)(T-\lambda)^{-1} \big)(I-E)\gamma}_{\ED H(\lambda)}, \,\,\, \lambda \in \rho(T) \cap \Lambda'.
\end{align}
Since
\begin{align*}
    \gamma_E(\lambda) \DE (I+(\lambda-\lambda_0)(T-\lambda)^{-1} \big)E\gamma =E\gamma (\lambda)
\end{align*}
maps into the Pontryagin space $E\mc K$, the function $Q$ belongs to $\mc N(\mc G)$ and has the realisation $(E\mc K,T|_{E \mc K},\gamma_E)$. We show that this realisation is minimal. Since $(\mc K,T,\gamma)$ is a minimal realisation of $Q$ over $\Lambda$, every $u=Eu \in E\mc K$ is the limit of some sequence $(u_n)_{n \in \bb N}$ from $\Span \big\{\gamma(\lambda) x \DS \lambda \in \rho(T) \cap \Lambda, \, x \in {\mc G} \big\}$. Since the range of the self-adjoint projection $E$ is a Pontryagin space, $E$ is bounded \cite[Proposition~5.1.1]{gheondea:2022}. Thus $Eu_n$ converges to $Eu=u$, and minimality of the realisation $(E\mc K,T|_{E \mc K},\gamma_E)$ is established.

We will now show that for each directional pair $(f_0;g_0)$ of $Q$ at $\omega$ there exists $v \in \mc K$ such that \eqref{O88}, \eqref{O89} hold. Let $f$ be a directional function corresponding to this directional pair. Since $H$ is holomorphic at $\omega$, it follows easily that $f$ is also a directional function of $Q_{\mc N}$ at $\omega$. The corresponding directional pair is $(f_0;g_0-H(\omega)f_0)$, and \Cref{O42} yields a vector $v_1 \in E\mc K$ such that
\begin{align*}
E\gamma {f_0} =\big(I+(\lambda_0-\omega)(T|_{E\mc K}-\lambda_0)^{-1} \big){v_1}, \\
{g_0}-H(\omega)f_0 = Q_{\mc N}(\lambda_0)^* {f_0} +(\omega-\overline{\lambda_0}) (E\gamma)^+ {v_1}.
\end{align*}
Since $\omega \in \rho(T|_{(I-E)\mc K})$ we can define
\begin{align*}
    v_2 &\DE \big(I+(\omega-\lambda_0)(T|_{(I-E)\mc K}-\omega)^{-1} \big)(I-E)\gamma f_0 \in (I-E)\mc K, \\
    v &\DE v_1+v_2.
\end{align*}
It follows that
\begin{align*}
    \gamma f_0 &=E\gamma f_0+(I-E)\gamma f_0=\big(I+(\lambda_0-\omega)(T|_{E\mc K}-\lambda_0)^{-1} \big){v_1} \\
    &+\big(I+(\lambda_0-\omega)(T|_{(I-E)\mc K}-\lambda_0)^{-1} \big){v_2}=\big(I+(\lambda_0-\omega)(T-\lambda_0)^{-1} \big){v}.
\end{align*}
Noting that $Q_{\mc N}(\lambda_0)^*=Q_{\mc N}(\overline{\lambda_0})=Q(\lambda_0)^*$ and $((I-E)\gamma)^+v_2=\frac{H(\omega)f_0}{\omega-\overline{\lambda_0}}$, we obtain
\begin{align*}
    g_0 &=Q_{\mc N}(\lambda_0)^* {f_0} +(\omega-\overline{\lambda_0}) (E\gamma)^+ {v_1}+H(\omega)f_0 \\
    &= Q(\lambda_0)^*+(\omega-\overline{\lambda_0})[(E\gamma)^+v_1+((I-E)\gamma)^+v_2] \\
    &=Q(\lambda_0)^*+(\omega-\overline{\lambda_0})\gamma^+v.
\end{align*}
Let us now prove the converse statement. Since $\omega \in \rho (T|_{(I-E)\mc K})$ we note that
    \begin{align*}
        \lim_{\lambda \hat{\to} \omega} (\lambda-\omega)(T-\lambda)^{-1}(I-E)\gamma f_0=0.
    \end{align*}
    Furthermore, our assumption is equivalent to $\omega$ not being an eigenvalue of $T|_{E\mc K}$. \Cref{O36} yields 
    \begin{align*}
        \lim_{\lambda \hat{\to} \omega} (\lambda-\omega)(T-\lambda)^{-1}E\gamma f_0=0.
    \end{align*}
    Together, this shows that 
     \begin{align*}
        \lim_{\lambda \hat{\to} \omega} (\lambda-\omega)(T-\lambda)^{-1}\gamma f_0=0.
    \end{align*}
    The rest of the proof is the same as that of \Cref{O43}.
\end{proof}

Before we continue, we note that for $Q \in \mc N(\mc G,\Omega)$ the well-definedness of $Q(\omega)$ is not as straightforward to show as for $Q \in \mc N(\mc G)$. Namely, the proof of \Cref{O91} uses the fact that an isometry between Pontryagin spaces with dense domain and range can be continued to a unitary operator, which is no longer true for isometries between Krein spaces. Although one could try to circumvent this problem with an argument involving \cite[Theorem~3.6]{jonas:2005}, well-definedness of $Q(\omega)$ is simply a by-product of the other statements of \Cref{O99} -- after all, the set of all directional pairs of $Q$ at $\omega$ is defined without reference to any realisation.

For the time being, however, we must take into account the possible dependence of $Q(\omega)$ on the realisation of $Q$. Hence, in the following proof (and only there) we write $Q_{(\mc K,T,\gamma)}(\omega)$ instead of $Q(\omega)$ to indicate  the choice of minimal $\pi_+$-realisation $(\mc K,T,\gamma)$ in \Cref{O87}. 

\begin{proof}[{Proof of \Cref{O99}}]
To begin with, note that (A) $\Leftrightarrow$ (A') is true by definition, and (A) $\Rightarrow$ (C) is precisely the first statement of \Cref{O84}. The proof of (B) $\Leftrightarrow$ (C) carries over verbatim from the proof of \Cref{O33}, as it consists only of algebraic manipulations. Hence we only need to show the implication (B) $\Rightarrow$ (A) and the uniqueness statement. We will use the following statements, which are all generalisations of results in \Cref{S3} that are proved in the same way as the original statements:
\begin{Enumerate}
    \item $\omega$ is an eigenvalue of $T$ if and only if there exists $g_0 \in \mc G \setminus \{0\}$ such that $(0;g_0) \in Q_{(\mc K,T,\gamma)}(\omega)$. 
    \item Assume that $f$ is a directional function of $Q$ at $\omega$ with directional pair $(f_0;g_0)$. Then it is also a directional function of $Q+\Theta$ for any self-adjoint $\Theta \in \mc B(\mc G)$, with directional pair $(f_0;g_0+\Theta f_0)$. If $Q$ is regular, then $\hat f(\lambda) \DE Q(\lambda)f(\lambda)$ is a directional function of $\wh Q$ at $\omega$ with directional pair $({g_0};-{f_0})$.
    \item Let $\Theta \in \mc B(\mc G)$ be self-adjoint. Then $(\mc K,T,\gamma)$ is also a minimal $\pi_+$-realisation of $Q+\Theta$ over $\Lambda$, and 
    \begin{align*}
        (Q+\Theta)_{(\mc K,T,\gamma)}(\omega)=Q_{(\mc K,T,\gamma)}(\omega)+\Theta.
    \end{align*}
    If $Q$ is regular then
    \begin{align*}
        \wh Q_{(\mc K,\wh T,\hat{\gamma})}(\omega)=-Q_{(\mc K,T,\gamma)}(\omega)^{-1}.
    \end{align*}
    \item There exists a self-adjoint $\Theta \in \mc B(\mc G)$ such that $0 \in \rho(Q(\lambda_0)+\Theta)$ and $\omega$ is not an eigenvalue of $\wh{T_\Theta}$, where
    \begin{align*}
        (\wh{T_\Theta}-\lambda_0)^{-1} \DE (T-\lambda_0)^{-1}+\gamma(\lambda_0)\widehat{\left( {Q+\Theta}\right)}(\lambda_0)\gamma(\overline{\lambda_0})^+.
    \end{align*}
\end{Enumerate}
With these auxiliary statements at hand we will show, step by step, that for any $(f_0;g_0) \in Q_{(\mc K,T,\gamma)}(\omega)$ there exists a directional function of $Q$ at $\alpha$ such that $(f_0;g_0)$ is the corresponding directional pair. The steps are as follows:
\begin{itemize}
    \item Choose $\Theta \in \mc B(\mc G)$ having the properties asserted in (iv).
    \item Transform $(f_0;g_0) \in Q_{(\mc K,T,\gamma)}(\omega)$ into $(-g_0-\Theta f_0;f_0) \in Q_{(\mc K,\wh{T_\Theta},\hat{\gamma}_{\Theta})}(\omega)$, using (iii). Here $\hat{\gamma}_{\Theta} \DE \gamma \wh{Q+\Theta}(\lambda_0)$.
    \item Due to (B) $\Leftrightarrow$ (C) there exists $v\in\mc K$ with
\begin{align*}
\hat{\gamma}_{\Theta} (-{g_0}-\Theta {f_0}) =\big(I+(\lambda_0-\omega)(\wh{T_\Theta}-\lambda_0)^{-1} \big){v}, \\
{f_0} = \wh{(Q+\Theta )}(\lambda_0)^* (-{g_0}-\Theta{f_0}) +(\omega-\overline{\lambda_0}) (\hat{\gamma}_{\Theta})^+ {v}.
\end{align*}
    \item By the converse statement in \Cref{O84}, the {constant} function $f_\Theta (\lambda) \DE-{g_0}-\Theta {f_0}$ is a directional function of $\wh{Q+\Theta}$ at $\omega$, with directional pair $(-g_0-\Theta f_0;f_0)$.
    \item Transform back using (ii): the directional function $f_\Theta$ of $\wh{Q+\Theta}$ transforms into a directional function of $Q$ with directional pair $(f_0;g_0)$.
\end{itemize}
These steps establish (B) $\Rightarrow$ (A). The uniqueness statement is shown in the same way as the uniqueness statement in \Cref{O33}.
\end{proof}

Also the characterisation of eigenvalues given in \Cref{O55} has a local analogue. The setting is as in \Cref{OS22}, except that $S$ now acts in a Krein space $\mc K$, and the self-adjoint extension $A$ (acting in $\mc K$ as well) is assumed to be of type $\pi_+$ over a domain $\Omega$ satisfying \eqref{O50}.

\begin{corollary}
\label{O5}
    Let $\wt A$ be a self-adjoint extension of $S$ in a Krein space $\wt{\mc K} \supseteq \mc K$. Assume that $\wt A$ is of type $\pi_+$ over $\Omega$, and that there exists $\tau \in \mc N(\mc G,\Omega)$ such that
    \begin{align}
P_{\mc K}(\wt A-\lambda)^{-1}|_{\mc K}=(A-\lambda)^{-1}-\gamma (\lambda)(m(\lambda)+\tau(\lambda))^{-1}\gamma(\overline{\lambda})^+
\end{align}
holds for all $\lambda \in \rho(A) \cap \mf h((m+\tau)^{-1})$. Then a point $\alpha \in \Omega$ is an eigenvalue of $\wt A$ if and only if there exist vectors $x_0,y_0\in\mathcal G$ with  $(x_0;y_0)\neq(0;0)$ such that 
\begin{equation}\label{eq:zeroes-main-krein}
   (x_0;y_0)\in m(\alpha) \quad \text{ and } \quad(x_0;-y_0) \in \tau(\alpha). 
\end{equation}
Furthermore, there  exists a linear bijection between $\ker (\wt A-\alpha)$ and the space of all $(x_0;y_0) \in \mc G \oplus \mc G$ satisfying \eqref{eq:zeroes-main-krein}.
\end{corollary}

\subsection{A special case: Local definiteness}
\label{OS63}

Our results from \Cref{sec-Nevanlinna} on functions in $\mc N_0(\mc G)$ continue to hold for functions that behave only locally like functions from $\mc N_0(\mc G)$ in the sense of the following definition.

\begin{definition}
\label{O95}
    Let $Q \in \mc N(\mc G,\Omega)$ and  the point $\omega \in \Omega \cap \bb R$. We then say that $Q$ is \emph{locally $\mc N_0$ at $\omega$} if there is a neighborhood $\mc U_\omega$ of $\omega$ where $Q$ can be decomposed as 
    \begin{align}
    \label{O47}
        Q(\lambda)=Q_0(\lambda)+H(\lambda), \qquad \lambda \in \mc U_\omega \cap \mf h(Q),
    \end{align}
    where $Q_0 \in \mc N_0(\mc G)$ and $H$ is holomorphic on $\mc U_\omega \cap \Omega$ with $H(\overline{\lambda})=H(\lambda)^*$.
\end{definition}

We point out that $Q \in \mc N(\mc G,\Omega)$ is locally $\mc N_0$ at $\omega$ if and only if $\omega$ is not a generalised pole of nonpositive type of $Q$, cf. \cite[Proposition~3.3]{daho.langer:1985}.

\begin{lemma}
\label{O98}
    Let $Q \in \mc N(\mc G,\Omega)$ be locally $\mc N_0$ at $\omega \in \Omega \cap \bb R$. Choose a decomposition of $Q$ as $Q_0+H$. Then
    \begin{align*}
        Q(\omega)=Q_0(\omega)+H(\omega).
    \end{align*}
Here the relations $Q(\omega)$ and $Q_0(\omega)$ are understood as generalised values, while $H(\omega) \in \mc B(\mc G)$ is just $H$ evaluated at $\omega$.
\end{lemma}
\begin{proof}
    Due to the characterisation in \Cref{O99}, we have $(f_0;g_0) \in Q(\omega)$ if and only if there exists a directional function $f$ of $Q$ at $\omega$ with directional pair $(f_0;g_0)$. Since $H$ is holomorphic at $\omega$, this is equivalent to $f$ being a directional function of $Q_0$ at $\omega$, with directional pair $(f_0;g_0-H(\omega)f_0)$. Again due to \Cref{O99}, this means that $(f_0;g_0-H(\omega)f_0) \in Q_0(\omega)$ or, alternatively, $(f_0;g_0) \in Q_0(\omega)+H(\omega)$.
\end{proof}

The decomposition $Q=Q_0+H$ of a function $Q$ that is locally $\mc N_0$ at $\omega$ is not unique. However, the measure $\Sigma$ in the integral representation \eqref{O96} of $Q_0$ is uniquely determined on compact subsets of $\mc U_\omega \cap \bb R$ by means of the Stieltjes inversion formula.

Using \Cref{O98} it becomes clear that the results of \Cref{sec-Nevanlinna} generalise to all functions in $\mc N(\mc G,\Omega)$ that are locally $\mc N_0$ at $\omega$. We state, in particular, the analogues of \Cref{O35} and \Cref{O11}. Proofs are straightforward and therefore omitted.
\begin{theorem}
    Let $Q \in \mc N(\mc G,\Omega)$ be locally $\mc N_0$ at $\omega \in \Omega \cap \bb R$. Then 
    \begin{align*}
        Q(\omega)=Q_{\emph{op}}(\omega) \dotplus (\{0\} \times \ran \Sigma (\{\omega \})^{\frac 12}).
    \end{align*}
    where
    \begin{align*}
        Q_{\emph{op}}(\omega) \DE \bigg\{ \big(f_0;\lim_{\lambda \hat{\to} \omega} Q(\lambda)f_0 \big) \DS \int_{\bb R} \frac{\DD (\Sigma(t)f_0,f_0)}{(t-\omega)^2}<\infty \bigg\}.
    \end{align*}
\end{theorem}

\begin{proposition}
\label{O14}
     Let $m,\tau \in \mc N(\mc G,\Omega)$ both be locally $\mc N_0$ at $\alpha \in \Omega \cap \bb R$. Then
     \begin{align*}
         (m+\tau)(\alpha)=m(\alpha)+\tau(\alpha).
     \end{align*}
\end{proposition}

\begin{remark}
    Just like the general eigenvalue criterion in \Cref{O55} carries over to the local case (cf. \Cref{O5}), also the simplified criterion given in \Cref{O52} for the definite situation stays valid locally. That is, if $\wt A$ and $m,\tau \in \mc N_0(\mc G,\Omega)$ are related through Krein's formula then a point $\alpha \in \Omega$ is an eigenvalue of $\wt A$ if and only if it is either a generalised zero of $m+\tau$, or a generalised pole of both $m$ and $\tau$ with a common pole vector. 
\end{remark}

\subsection*{Acknowledgements.}
We would like to thank Seppo Hassi for sharing the unpublished manuscript \cite{hassi:spexit} with us, and his kind communication on the topic that helped us to identify a problem in \Cref{sec-Nevanlinna}. \newline


{\footnotesize

\begin{flushleft}
	A.~Luger \\
	Department of Mathematics\\
	Stockholms universitet \\
    106 91 Stockholm \\
	SWEDEN \\
	email: luger@math.su.se \\[5mm]
\end{flushleft}

\begin{flushleft}
	J.~Reiffenstein \\
	Department of Mathematics\\
	Stockholms universitet \\
    106 91 Stockholm \\
	SWEDEN \\
	email: jakob.reiffenstein@math.su.se \\[5mm]
\end{flushleft}

}

\end{document}